\documentclass[10pt]{article}

\usepackage[margin=1in]{geometry}
\usepackage{amsmath, amsthm, mathtools}
\usepackage{newpxtext, newpxmath}
\usepackage{microtype}

\usepackage{etoolbox}
\usepackage{enumitem}

\usepackage{xcolor}
\usepackage{tikz}
\tikzset{
  vertex/.style={circle, draw, fill=white, inner sep=1.5pt},  % open vertex
  wild edge/.style={thick, black},                            % wild (W) edges
  cedge/.style={thick, dashed, blue},                         % edges of color c_i
}

\usepackage{standalone}
\usepackage{float}

\usepackage{hyperref}
\hypersetup{
    colorlinks=true, 
    linkcolor=blue!50!black,
    citecolor=green!40!black,
    urlcolor=blue!60!black,
}
\usepackage[capitalize]{cleveref}

\newcounter{case}
\crefname{case}{Case}{Cases}
\crefname{construction}{Construction}{Constructions}
\newcommand{\case}[1]{%
    \par\addvspace{\topsep}%
    \refstepcounter{case}%
    \noindent\textbf{Case~\thecase.}~#1.%
    \enspace\ignorespaces
}
\AtBeginEnvironment{proof}{\setcounter{case}{0}}

\DeclareMathOperator{\wld}{wld} % wild number
\DeclareMathOperator{\clb}{clb} % component lower bound

\DeclarePairedDelimiter{\floor}{\lfloor}{\rfloor}
\DeclarePairedDelimiter{\ceil}{\lceil}{\rceil}

\newcommand{\CC}{\mathcal{C}} % color set
\newcommand{\E}{\mathbb{E}}   % expectation
\newcommand{\rest}[1]{|_{#1}} % restriction of a coloring
\newcommand{\Wg}{W_{\mathrm{greedy}}} % greedy output

\usepackage{thmtools}
\declaretheorem[style=plain, numberwithin=section]{theorem}
\declaretheorem[style=plain, sibling=theorem]{lemma, proposition, corollary, conjecture}
\declaretheorem[style=definition, sibling=theorem]{definition, construction}

\numberwithin{figure}{section}

\title{Wild number for spanning trees in complete and complete bipartite graphs}

\author{Aleyah Dawkins \footnote{Carnegie Mellon University, \texttt{adawkins@andrew.cmu.edu}.}
\and
Alexander Nappo \footnote{Villanova University}
\and
Jessie Tong \footnote{San Diego State University}
\and Christopher Zimbizi \footnote{Lake Forest College}}

\date{October 15, 2026}

\begin{document}

\maketitle

\begin{abstract}
    Given an edge-colored graph, the wild number measures how close the graph is to having a spanning tree in every color. We investigate this parameter for complete graphs and complete bipartite graphs, with particular emphasis on minimum proper colorings and equitable $2$-colorings. For complete graphs, we determine the full range of wild numbers attainable by equitable $2$-colorings and show that, although every value in this range occurs, the expected wild number of a uniformly random equitable $2$-coloring tends to $0$ as the order of the graph grows. For minimum proper colorings of $K_n$, we establish a general lower bound, determine the exact wild number for $3 \le n \le 12$, and conjecture that the lower bound is sharp for all $n$. We also develop structural conditions on ideal wild sets and use them to obtain upper bounds and analyze the performance of a natural greedy algorithm. For complete bipartite graphs, we prove a sharp upper bound for equitable $2$-colorings and determine the wild number exactly for minimum proper colorings. Finally, we apply these results to Cartesian products of complete graphs and complete bipartite graphs, obtaining lower bounds for the wild number under minimum proper and product colorings.
\end{abstract}

\section{Introduction}

Since their introduction in 1961 by Nash-Williams and Tutte \cite{NashWilliams,Tutte}, edge-disjoint spanning trees have been instrumental in studying network reliability and performance. The existence of multiple edge-disjoint spanning trees provides independent, yet parallel, paths between vertices, allowing a network to withstand edge failures and avoid bottlenecks. These independent pathways are especially important in networks that wish to support collective communication operations, as they can be used to enhance fault tolerance while managing congestion. Independently, Nash-Williams and Tutte proved that a connected graph $G$ has $k$ edge-disjoint spanning trees if and only if, for any partition $P$ of the vertices of $G$, there are at least $k(|P|-1)$ crossing edges. Determining conditions that guarantee edge-disjoint spanning trees in various graphs has been studied extensively and remains an active area of research \cite{Tarjan,CIOABA,EDSTstars}. 

Edge colorings provide a natural way to organize and allocate the pathways of a network across distinct operations. Tait \cite{tait1880} showed that the four-color problem is equivalent to $3$-edge coloring every planar $3$-connected cubic graph, and edge colorings of many types have since been defined and studied \cite{vizing1964,Grnbaum1973Acyclic,FouquetJolivet1983,Hilton1994}. We refer the reader to a recent survey on edge colorings by Cao, Chen, Jing, Stiebitz, and Toft \cite{cao2019survey} for a more detailed overview. The applications of edge colorings are far-reaching, including scheduling, routing, and resource allocation \cite{sportssched,GANDHAM20081122,LevRouting,Barrett}.

These two topics---edge-disjoint spanning trees as a gauge of network reliability and edge colorings as a means of organizing network operations---come together naturally in the study of the wild number of an edge-colored graph. In network terms, it is the minimum number of links that must be upgraded to carry every operation so that each operation has its own spanning tree. This idea, first introduced and explored by Anders, Foster-Greenwood, Garcia, and Krawzik~\cite{WildNumber}, was inspired by the algebraic theory of generalized graph splines. See \cite{WildSplines} for more on the connection between graphs that admit only constant splines and the wild number of a graph. The Nash-Williams--Tutte theorem can then be thought of as a characterization of when a graph has a $k$-edge coloring that produces wild number $0$. 

In this paper, we investigate the wild number of the complete graph and the complete bipartite graph for minimum proper edge colorings and equitable $2$-edge colorings. Complete graphs are a natural setting for many edge coloring problems, for example, all-to-all scheduling. The complete bipartite graph has applications to scheduling as well, where the interest is the possible interactions between two groups. Minimum proper colorings offer the shortest possible conflict-free schedule for operations that cannot be performed simultaneously, and equitable $2$-colorings offer insight into distributing workload evenly between two operations to prevent congestion.

We begin with preliminary definitions and results in \cref{sec:prelims}. In \cref{sec:complete}, we investigate complete graphs, determining the range of possible wild numbers for equitable $2$-colorings, along with the expected wild number of a uniformly random equitable $2$-coloring. We also provide a lower bound on the wild number for minimum proper colorings and verify that this lower bound is sharp for $3 \le n \le 12$, motivating a conjecture that it is sharp for all $n$, as well as an analysis of the performance of a natural greedy algorithm for constructing these wild sets. We shift our focus to complete bipartite graphs in \cref{sec:complete-bipartite}, proving a sharp upper bound on the wild number for equitable $2$-colorings and determining the wild number exactly for minimum proper colorings. In \cref{sec:products}, we use our results from \cref{sec:complete,sec:complete-bipartite} on minimum proper colorings to obtain lower bounds for Cartesian products of complete graphs and of complete bipartite graphs. We conclude with some remarks in \cref{sec:conclusion}.

\section{Preliminaries}\label{sec:prelims}

Let $\CC = \{c_1, \ldots, c_\ell\}$ be a set of colors. An \emph{$\ell$-edge coloring} of a graph $G$ is a function $\gamma: E(G) \rightarrow \CC$ that assigns a color to each edge of $G$. We call the edge set of a particular color a \emph{color class}. A connected graph $G$ with an edge coloring $\gamma$ is called an \emph{edge-colored graph}, denoted $(G,\gamma)$. When context is clear, we refer to an edge-colored graph $(G,\gamma)$ simply as $G$. Throughout this paper, every edge coloring is assumed to be surjective.

Let $W \subseteq E(G)$. For each color $c_i$, let $E(G_i)=\{e  \in E(G):\gamma(e)=c_i\}$, and let $G_i$ be the spanning subgraph of $G$ with edge set $E(G_i)$. Define $G^W_i$ to be the spanning subgraph of $G$ with edge set $E(G_i) \cup W$. If $G_i^W$ is connected for every $c_i \in \CC$, then $W$ is a \emph{color-connecting set}. We say that an edge $e$ is colored \emph{wild} if $e \in W$. A color-connecting set of minimum size is an \emph{ideal wild set}. The \emph{wild number} of $G$ with respect to the coloring $\gamma$ is 
$$\wld_\gamma(G):=\min\{|W|:W \text{ is a color-connecting set for } G \text{ with coloring } \gamma\}.$$

When the edge coloring is clear from context, we abbreviate $\wld_\gamma(G)$ by $\wld(G)$. If $\ell=1$, then $\wld_\gamma(G)=0$, so we assume $\ell\ge 2$ throughout. We focus our investigation on two commonly studied edge colorings, minimum proper and equitable. A \emph{proper edge coloring} assigns a color to each edge of $G$ such that adjacent edges receive different colors. A proper coloring of $G$ is \emph{minimum} if it uses exactly $\chi'(G)$ colors. An edge coloring is \emph{equitable} if the sizes of any two color classes differ by at most one. Note that it is not required that an equitable coloring be proper. For any $1 \leq i \leq \ell$, we use the notation $\kappa(G_i)$ for the number of connected components of the monochromatic subgraph $G_i$. For general graphs $G$, we use $\kappa(G)$ to denote the number of connected components in $G$ and $\kappa^*(G)$ to denote the number of nontrivial connected components in $G$, that is, components containing at least one edge. 

By the definition of the wild number, for any coloring $\gamma$ of an $n$-vertex graph $G$, we have that 
\begin{equation}\label{eq:def-bound}
    0 \le \wld_\gamma(G) \le n-1.
\end{equation} The following results allow the wild number of a graph to be bounded by the number of monochromatic components.

\begin{proposition}[{\cite[Proposition~5.1]{WildNumber}}]\label{prop:comp-bound}
    If $(G, \gamma)$ is an edge-colored graph with $\ell$ colors, then $$\max_{1 \le i \le \ell} (\kappa (G_i) - 1) \le \wld_\gamma(G) \le \sum_{i=1}^\ell (\kappa (G_i) - 1).$$
\end{proposition}

\begin{proposition}[{\cite[Proposition~5.4]{WildNumber}}]\label{prop:ceil-bound}
    If $(G, \gamma)$ is an edge-colored graph with $\ell$ colors, then $$\wld_\gamma(G) \ge \ceil*{\frac{\sum_{i=1}^\ell (\kappa (G_i) - 1)}{\ell-1}}.$$
\end{proposition}

For $\ell=2$, the upper bound in \cref{prop:comp-bound} and the bound in \cref{prop:ceil-bound} coincide, so every $2$-edge-colored graph satisfies \cite[Theorem~6.4(2)]{WildNumber}
\begin{equation}\label{eq:two-colors}
    \wld(G)=(\kappa(G_1)-1)+(\kappa(G_2)-1).
\end{equation}

\section{Complete Graphs}\label{sec:complete}
In this section, we explore the wild number of the complete graph given particular edge colorings. We first consider $2$-colorings, then restrict to equitable $2$-colorings before finally considering minimum proper colorings. 

\subsection{$2$-Colorings}
A classical theorem of Nash-Williams and Tutte \cite{NashWilliams,Tutte} implies that the spanning tree packing number of $K_n$ is $\floor{n/2}$. For $n \geq 4$, this yields at least $2$ edge-disjoint spanning trees in $K_n$; call them $T_1$ and $T_2$. Let $\gamma: E(K_n) \rightarrow \{c_1,c_2\}$ be an edge coloring such that $E(T_1) \subseteq E(G_1)$ and $E(T_2) \subseteq E(G_2)$, where the remaining edges of $K_n$ are colored arbitrarily. Then both $G_1$ and $G_2$ are connected spanning subgraphs. Thus, by \eqref{eq:two-colors}, $\wld_\gamma(K_n) = (1-1) + (1-1)  = 0$.
    
We now construct a $2$-coloring producing a color class with $r$ components.

\begin{construction}[Clique-fragmentation coloring]\label{con:clique-frag}
    Let $n>r\ge 2$. Fix a partition of $n$ into $r$ positive integers $s_1 + s_2 + \dots + s_r = n$. Partition $V(K_n)$ into disjoint, nonempty sets $V_1, \dots, V_r$ with $|V_i| = s_i$. Define a $2$-coloring such that:

\begin{itemize}
    \item $G_1 = \bigcup_{i=1}^r K_{s_i}$ (a disjoint union of cliques),
    \item $G_2 = K_{s_1, s_2, \dots, s_r}$ (the complement, a complete multipartite graph).
\end{itemize}
\end{construction}

The condition $r \geq 2$ ensures $E(G_2)$ is nonempty, and $n>r$ ensures $E(G_1)$ is nonempty, guaranteeing surjectivity. An example of such a construction is shown in \cref{fig:k12-con1} where $n=12$, $r=3$, and $s_i=4$ for all $i$.

\begin{figure}[ht]
    \centering
    \begin{tikzpicture}[scale=3, every node/.style={circle, draw, fill=white, inner sep=1pt}]

\begin{scope}[shift = {(0,0)}]
  % Define vertex positions on a circle
  \node (v1) at (0:1) {$v_1$};
  \node (v2) at (30:1) {$v_2$};
  \node (v3) at (60:1) {$v_3$};
  \node (v4) at (90:1) {$v_4$};
  \node (v5) at (120:1) {$v_5$};
  \node (v6) at (150:1) {$v_6$};
  \node (v7) at (180:1) {$v_7$};
  \node (v8) at (210:1) {$v_8$};
  \node (v9) at (240:1) {$v_9$};
  \node (v10) at (270:1) {$v_{10}$};
  \node (v11) at (300:1) {$v_{11}$};
  \node (v12) at (330:1) {$v_{12}$};

  %-----------------------------
  % Red edges
  %-----------------------------

  \draw[line width = 2pt, red] (v1) -- (v2);
  \draw[line width = 2pt, red] (v1) -- (v3);
  \draw[line width = 2pt, red] (v1) -- (v4);
  \draw[line width = 2pt, red] (v2) -- (v3);
  \draw[line width = 2pt, red] (v2) -- (v4);
  \draw[line width = 2pt, red] (v3) -- (v4);

  \draw[line width = 2pt, red] (v5) -- (v6);
  \draw[line width = 2pt, red] (v5) -- (v7);
  \draw[line width = 2pt, red] (v5) -- (v8);
  \draw[line width = 2pt, red] (v6) -- (v7);
  \draw[line width = 2pt, red] (v6) -- (v8);
  \draw[line width = 2pt, red] (v7) -- (v8);

  \draw[line width = 2pt, red] (v9) -- (v10);
  \draw[line width = 2pt, red] (v9) -- (v11);
  \draw[line width = 2pt, red] (v9) -- (v12);
  \draw[line width = 2pt, red] (v10) -- (v11);
  \draw[line width = 2pt, red] (v10) -- (v12);
  \draw[line width = 2pt, red] (v11) -- (v12);

\end{scope}

%%%%%% JUST ONE COLOR CLASS %%%%%%%%%

\begin{scope}[shift = {(2.5,0)}]
  % Define vertex positions on a circle
  \node (v1) at (0:1) {$v_1$};
  \node (v2) at (30:1) {$v_2$};
  \node (v3) at (60:1) {$v_3$};
  \node (v4) at (90:1) {$v_4$};
  \node (v5) at (120:1) {$v_5$};
  \node (v6) at (150:1) {$v_6$};
  \node (v7) at (180:1) {$v_7$};
  \node (v8) at (210:1) {$v_8$};
  \node (v9) at (240:1) {$v_9$};
  \node (v10) at (270:1) {$v_{10}$};
  \node (v11) at (300:1) {$v_{11}$};
  \node (v12) at (330:1) {$v_{12}$};
  
  %-----------------------------
  % Blue edges
  %-----------------------------

  \draw[line width = 2pt, blue] (v1) -- (v5);
  \draw[line width = 2pt, blue] (v1) -- (v6);
  \draw[line width = 2pt, blue] (v1) -- (v7);
  \draw[line width = 2pt, blue] (v1) -- (v8);
  \draw[line width = 2pt, blue] (v1) -- (v9);
  \draw[line width = 2pt, blue] (v1) -- (v10);
  \draw[line width = 2pt, blue] (v1) -- (v11);
  \draw[line width = 2pt, blue] (v1) -- (v12);

  \draw[line width = 2pt, blue] (v2) -- (v5);
  \draw[line width = 2pt, blue] (v2) -- (v6);
  \draw[line width = 2pt, blue] (v2) -- (v7);
  \draw[line width = 2pt, blue] (v2) -- (v8);
  \draw[line width = 2pt, blue] (v2) -- (v9);
  \draw[line width = 2pt, blue] (v2) -- (v10);
  \draw[line width = 2pt, blue] (v2) -- (v11);
  \draw[line width = 2pt, blue] (v2) -- (v12);

  \draw[line width = 2pt, blue] (v3) -- (v5);
  \draw[line width = 2pt, blue] (v3) -- (v6);
  \draw[line width = 2pt, blue] (v3) -- (v7);
  \draw[line width = 2pt, blue] (v3) -- (v8);
  \draw[line width = 2pt, blue] (v3) -- (v9);
  \draw[line width = 2pt, blue] (v3) -- (v10);
  \draw[line width = 2pt, blue] (v3) -- (v11);
  \draw[line width = 2pt, blue] (v3) -- (v12);

  \draw[line width = 2pt, blue] (v4) -- (v5);
  \draw[line width = 2pt, blue] (v4) -- (v6);
  \draw[line width = 2pt, blue] (v4) -- (v7);
  \draw[line width = 2pt, blue] (v4) -- (v8);
  \draw[line width = 2pt, blue] (v4) -- (v9);
  \draw[line width = 2pt, blue] (v4) -- (v10);
  \draw[line width = 2pt, blue] (v4) -- (v11);
  \draw[line width = 2pt, blue] (v4) -- (v12);

  \draw[line width = 2pt, blue] (v5) -- (v9);
  \draw[line width = 2pt, blue] (v5) -- (v10);
  \draw[line width = 2pt, blue] (v5) -- (v11);
  \draw[line width = 2pt, blue] (v5) -- (v12);

  \draw[line width = 2pt, blue] (v6) -- (v9);
  \draw[line width = 2pt, blue] (v6) -- (v10);
  \draw[line width = 2pt, blue] (v6) -- (v11);
  \draw[line width = 2pt, blue] (v6) -- (v12);

  \draw[line width = 2pt, blue] (v7) -- (v9);
  \draw[line width = 2pt, blue] (v7) -- (v10);
  \draw[line width = 2pt, blue] (v7) -- (v11);
  \draw[line width = 2pt, blue] (v7) -- (v12);

  \draw[line width = 2pt, blue] (v8) -- (v9);
  \draw[line width = 2pt, blue] (v8) -- (v10);
  \draw[line width = 2pt, blue] (v8) -- (v11);
  \draw[line width = 2pt, blue] (v8) -- (v12);

\end{scope}

\end{tikzpicture}
    \caption{An example of \cref{con:clique-frag}.}
    \label{fig:k12-con1}
\end{figure}

The following lemma shows that the wild number of $K_n$ with a clique-fragmentation coloring is determined by the number of parts in
the vertex partition of $K_n$.

\begin{lemma}\label{lem:clique-frag}
For any clique-fragmentation coloring with $r$ parts, we have  $$\wld(K_n) = r - 1.$$
\end{lemma}

\begin{proof}
    By construction, $G_1$ consists of $r$ vertex-disjoint cliques, so $\kappa(G_1) = r$. Since $r \geq 2$ and $s_i \ge 1$ for all $i$, the complete multipartite graph $G_2$ is connected. Applying \eqref{eq:two-colors}, we have
    $$
        \wld(K_n) 
        = 
        \kappa(G_1)+\kappa(G_2)-2 = 
        r+1-2 
        =r-1.\qedhere
    $$
\end{proof}

\begin{theorem}\label{thm:bounds}
    Let $n \ge 4$. For any $2$-coloring of $K_n$,
    $$
        0 \leq \wld(K_n) \leq n - 2.
    $$ Moreover, every integer in this range is attainable by some $2$-coloring. 
\end{theorem}

\begin{proof}
    The surjectivity of the coloring guarantees that $\kappa(G_1) \geq 1$ and $\kappa(G_2) \geq 1$. For any graph $H$ on $n \geq 2$ vertices, at least one of $H$ or its complement $\overline{H}$ is connected. Since $G_1$ and $G_2$ are complements of each other, we then have $\kappa(G_1)=1$ or $\kappa(G_2) = 1$, say $\kappa(G_1) = 1$. A graph on $n$ vertices containing at least one edge has at most $n-1$ connected components, so $$\kappa(G_1) + \kappa(G_2) \leq 1 + (n-1) = n.$$ From \cref{prop:comp-bound}, it follows that $0 \leq \wld(K_n) \leq n - 2$. Since $n \ge 4$, the attainability of the wild number $0$ is shown by the coloring discussed at the beginning of this subsection. For $1\leq k\leq n-2$, choose $r=k+1$ in \cref{con:clique-frag}. Then, by \cref{lem:clique-frag}, we have $\wld(K_n)=r-1=k$, verifying the attainability of every wild number in the range $\{1, \dots, n-2\}$.
\end{proof}

\subsection{Equitable $2$-Colorings}\label{sec:complete-eq2}

We determine the set of wild numbers attainable by equitable $2$-colorings of $K_n$. We begin by constructing an equitable $2$-coloring that can be defined to achieve any number in the full range of possible wild numbers. 

\begin{construction}\label{con:eq2-kn}
Let $n \ge 4$ and let $m \in \{\floor{n(n-1)/4}, \ceil{n(n-1)/4}\}$. Set $t=\ceil*{\frac{1+\sqrt{1+8m}}{2}}$. Fix $m$ and let $1 \le k \le n-t$. We construct an equitable $2$-coloring such that there is a color class containing $m$ edges that form $k+1$ components consisting of one component of size $n-k$ and the remaining components consisting of isolated vertices. 

Partition $V(K_n)$ into two sets $U_1$ and $V_1$ such that $|U_1|=k$ and $|V_1|=n-k$. Choose a spanning tree of $K_n[V_1]$, and then add further edges of $K_n[V_1]$ until the subgraph has $m$ edges. Existence of $m$ edges is guaranteed since $k \le n-t$ gives $n-k \ge t$, and so $\binom{n-k}{2} \ge \binom{t}{2}\ge m$. Connectedness is guaranteed, since $n \ge 4$, by $m\ge \floor{n(n-1)/4} \ge n-2 \ge n-k-1$. 

Let this set of $m$ edges be $M \subset E(K_n)$. Define a $2$-coloring $\gamma$ such that:

\begin{itemize}
    \item $G_1$ is the graph with vertex set $V(K_n)$ and edge set $M$.
    \item $G_2$ is the graph with vertex set $V(K_n)$ and edge set $E(K_n) \setminus M$.
\end{itemize}
\end{construction}

Note that $\kappa(G_1)=k+1$ and $\kappa(G_2)=1$ since $G_2$ contains every edge between $U_1$ and $V_1$. An example of \cref{con:eq2-kn} is shown in \cref{fig:k12-equitable} for $n=12$, $m=33$, and $k=n-t=3$.

\begin{figure}[ht]
    \centering
    \begin{tikzpicture}[scale=3, every node/.style={circle, draw, fill=white, inner sep=1pt}]

\begin{scope}[shift = {(0,0)}]
  % Define vertex positions on a circle
  \node (v1) at (0:1) {$v_1$};
  \node (v2) at (30:1) {$v_2$};
  \node (v3) at (60:1) {$v_3$};
  \node (v4) at (90:1) {$v_4$};
  \node (v5) at (120:1) {$v_5$};
  \node (v6) at (150:1) {$v_6$};
  \node (v7) at (180:1) {$v_7$};
  \node (v8) at (210:1) {$v_8$};
  \node (v9) at (240:1) {$v_9$};
  \node (v10) at (270:1) {$v_{10}$};
  \node (v11) at (300:1) {$v_{11}$};
  \node (v12) at (330:1) {$v_{12}$};

  \node[draw=none, fill=none] at (0,1.3) {$G_1$};

  %-----------------------------
  % Blue edges
  %-----------------------------

  \draw[line width = 2pt, blue] (v4) -- (v8);
  \draw[line width = 2pt, blue] (v4) -- (v9);
  \draw[line width = 2pt, blue] (v4) -- (v10);
  \draw[line width = 2pt, blue] (v4) -- (v11);
  \draw[line width = 2pt, blue] (v4) -- (v12);

  \draw[line width = 2pt, blue] (v5) -- (v6);
  \draw[line width = 2pt, blue] (v5) -- (v7);
  \draw[line width = 2pt, blue] (v5) -- (v8);
  \draw[line width = 2pt, blue] (v5) -- (v9);
  \draw[line width = 2pt, blue] (v5) -- (v10);
  \draw[line width = 2pt, blue] (v5) -- (v11);
  \draw[line width = 2pt, blue] (v5) -- (v12);

  \draw[line width = 2pt, blue] (v6) -- (v7);
  \draw[line width = 2pt, blue] (v6) -- (v8);
  \draw[line width = 2pt, blue] (v6) -- (v9);
  \draw[line width = 2pt, blue] (v6) -- (v10);
  \draw[line width = 2pt, blue] (v6) -- (v11);
  \draw[line width = 2pt, blue] (v6) -- (v12);

  \draw[line width = 2pt, blue] (v7) -- (v8);
  \draw[line width = 2pt, blue] (v7) -- (v9);
  \draw[line width = 2pt, blue] (v7) -- (v10);
  \draw[line width = 2pt, blue] (v7) -- (v11);
  \draw[line width = 2pt, blue] (v7) -- (v12);

  \draw[line width = 2pt, blue] (v8) -- (v9);
  \draw[line width = 2pt, blue] (v8) -- (v10);
  \draw[line width = 2pt, blue] (v8) -- (v11);
  \draw[line width = 2pt, blue] (v8) -- (v12);

  \draw[line width = 2pt, blue] (v9) -- (v10);
  \draw[line width = 2pt, blue] (v9) -- (v11);
  \draw[line width = 2pt, blue] (v9) -- (v12);

  \draw[line width = 2pt, blue] (v10) -- (v11);
  \draw[line width = 2pt, blue] (v10) -- (v12);
  
  \draw[line width = 2pt, blue] (v11) -- (v12);

\end{scope}

%%%%%% JUST ONE COLOR CLASS %%%%%%%%%

\begin{scope}[shift = {(2.5,0)}]
  % Define vertex positions on a circle
  \node (v1) at (0:1) {$v_1$};
  \node (v2) at (30:1) {$v_2$};
  \node (v3) at (60:1) {$v_3$};
  \node (v4) at (90:1) {$v_4$};
  \node (v5) at (120:1) {$v_5$};
  \node (v6) at (150:1) {$v_6$};
  \node (v7) at (180:1) {$v_7$};
  \node (v8) at (210:1) {$v_8$};
  \node (v9) at (240:1) {$v_9$};
  \node (v10) at (270:1) {$v_{10}$};
  \node (v11) at (300:1) {$v_{11}$};
  \node (v12) at (330:1) {$v_{12}$};

  \node[draw=none, fill=none] at (0,1.3) {$G_2$};

  %-----------------------------
  % Red edges
  %-----------------------------

  \draw[line width = 2pt, red] (v1) -- (v2);
  \draw[line width = 2pt, red] (v1) -- (v3);
  \draw[line width = 2pt, red] (v1) -- (v4);
  \draw[line width = 2pt, red] (v1) -- (v5);
  \draw[line width = 2pt, red] (v1) -- (v6);
  \draw[line width = 2pt, red] (v1) -- (v7);
  \draw[line width = 2pt, red] (v1) -- (v8);
  \draw[line width = 2pt, red] (v1) -- (v9);
  \draw[line width = 2pt, red] (v1) -- (v10);
  \draw[line width = 2pt, red] (v1) -- (v11);
  \draw[line width = 2pt, red] (v1) -- (v12);

  \draw[line width = 2pt, red] (v2) -- (v3);
  \draw[line width = 2pt, red] (v2) -- (v4);
  \draw[line width = 2pt, red] (v2) -- (v5);
  \draw[line width = 2pt, red] (v2) -- (v6);
  \draw[line width = 2pt, red] (v2) -- (v7);
  \draw[line width = 2pt, red] (v2) -- (v8);
  \draw[line width = 2pt, red] (v2) -- (v9);
  \draw[line width = 2pt, red] (v2) -- (v10);
  \draw[line width = 2pt, red] (v2) -- (v11);
  \draw[line width = 2pt, red] (v2) -- (v12);

  \draw[line width = 2pt, red] (v3) -- (v4);
  \draw[line width = 2pt, red] (v3) -- (v5);
  \draw[line width = 2pt, red] (v3) -- (v6);
  \draw[line width = 2pt, red] (v3) -- (v7);
  \draw[line width = 2pt, red] (v3) -- (v8);
  \draw[line width = 2pt, red] (v3) -- (v9);
  \draw[line width = 2pt, red] (v3) -- (v10);
  \draw[line width = 2pt, red] (v3) -- (v11);
  \draw[line width = 2pt, red] (v3) -- (v12);

  \draw[line width = 2pt, red] (v4) -- (v5);
  \draw[line width = 2pt, red] (v4) -- (v6);
  \draw[line width = 2pt, red] (v4) -- (v7);

\end{scope}

\end{tikzpicture}
    \caption{\cref{con:eq2-kn}, with $n=12$ and $m=33$, attaining the upper bound for equitable $2$-colorings: $\wld_\varepsilon(K_{12})=12-\ceil{(1+\sqrt{1+8\cdot33})/2}=3$.}
    \label{fig:k12-equitable}
\end{figure}

\begin{lemma}\label{lem:eq2-kn}
    Let $n \ge 4$ and let $m \in \{\floor{n(n-1)/4}, \ceil{n(n-1)/4}\}$. Fix $m$ and let $1 \le k \le n-\left\lceil\frac{1+\sqrt{1+8m}}{2}\right\rceil$. \cref{con:eq2-kn} defines an equitable $2$-coloring $\varepsilon$ such that $$\wld_\varepsilon(K_n)=k.$$
\end{lemma}

\begin{proof}
     By construction, $G_1$ consists of $k+1$ components and $G_2$ is connected. From \eqref{eq:two-colors}, as $\varepsilon$ is a $2$-coloring, it follows that $\wld_\varepsilon(K_n) = \kappa(G_1) + \kappa(G_2)-2=(k+1)+1-2=k$.
\end{proof}

\begin{theorem}\label{thm:eq2-kn}
    Let $n \ge 4$. For any equitable $2$-coloring of the edges of $K_n$, if both color classes are connected, then $$\wld(K_n)=0.$$ Otherwise, let $m\in \{\floor{n(n-1)/4}, \ceil{n(n-1)/4}\}$ be the number of edges in a disconnected color class. Then $$1 \le \wld(K_n)\le n-\ceil*{\frac{1+\sqrt{1+8m}}{2}}.$$ Moreover, every integer from $0$ to $n-\ceil*{\frac{1+\sqrt{1+8m}}{2}}$ is attainable by some equitable $2$-coloring. 
\end{theorem}

\begin{proof}
     By \cref{prop:comp-bound}, we have 
     \begin{equation}\label{eq:2col-bound}
         \wld(K_n)\le \kappa (G_1)+\kappa(G_2)-2.
     \end{equation} 
     Since $G_2=\overline{G_1}$, at most one of $\{G_1,G_2\}$ is disconnected. If both are connected, then \eqref{eq:2col-bound} gives $\wld(K_n)=0$. For the remainder of the proof, we may assume without loss of generality that $\kappa(G_1)>1$ and $\kappa(G_2)=1$. Set $m=|E(G_1)|$; then $m \in \{\floor{n(n-1)/4}, \ceil{n(n-1)/4}\}$ and $|E(G_2)|=\binom{n}{2} - m$. Now we bound the number of components in $G_1$. 
     
     Let $\kappa(G_1)=q$ with components of size $a_1,a_2,\dots,a_q$ where $\sum_{i=1}^qa_i=n$ and each $a_i>0$. The maximum number of edges in a component on $a_i$ vertices is $\binom{a_i}{2}$, so \begin{align*}
     m\le \sum_{i=1}^q\binom{a_i}{2}.
     \end{align*} Since $x\mapsto\binom{x}{2}$ is convex, this sum is maximized when the component sizes are as unequal as possible, namely when one component contains $n-q+1$ vertices and the remaining $q-1$ components are isolated vertices. Thus, $$m\leq \sum_{i=1}^q\binom{a_i}{2}
    \leq \binom{n-q+1}{2}.$$ Let $t=\min\{s : \binom{s}{2}\ge m\}$. Then $t\le n-q+1$ and $q\le n-t+1$. Since $t=\min\{s : \binom{s}{2}\ge m\}$, we have $\binom{t}{2}=\frac{t(t-1)}{2} \ge m$ and thus $t\ge \frac{1+\sqrt{1+8m}}{2}$. By minimality, it follows that $t=\ceil*{\frac{1+\sqrt{1+8m}}{2}}$. From \eqref{eq:2col-bound}, we have $$\wld(K_n)\le q-1\le n-t = n-\ceil*{\frac{1+\sqrt{1+8m}}{2}}.$$
     
     Recall that the spanning tree packing number of $K_n$ is $\floor{n/2}$, so for $n\ge 4$ there are at least two edge-disjoint spanning trees in $K_n$. Assigning a spanning tree to each color class and partitioning the remaining edges in an equitable way verifies the attainability of wild number $0$. For all $1 \le k \le n-\ceil*{\frac{1+\sqrt{1+8m}}{2}}$, \cref{con:eq2-kn} and \cref{lem:eq2-kn} ensure that $\wld_\varepsilon(K_n)= k$, verifying the attainability of every wild number in the range of $\{1, \dots, n-\ceil*{\frac{1+\sqrt{1+8m}}{2}}\}$.
\end{proof}
Note that for all $n \ge 4$, we can simplify the bound to $0 \le \wld(K_n) \le (1-\sqrt{2}/2)\,n+3/2$. 

\subsubsection{Expected Wild Number}

As shown above, every wild number in the attainable range for an equitable $2$-coloring of $K_n$ occurs. However, the expected wild number for a uniformly random equitable $2$-coloring of $K_4$ is only $2/5$. We show that, as $n$ grows, the expected wild number of a uniformly random equitable $2$-coloring of $K_n$ tends to $0$. The following lemma is an immediate consequence of Gilbert's bounds on the connectivity probability.

\begin{lemma}[\cite{Gilbert}]
\label{lem:gilbert-connectivity}
Let $G$ be the random graph on $n$ vertices where every edge occurs independently with probability $1/2$. Then $\Pr[G\text{ is disconnected}]=O(n2^{-n})$.
\end{lemma}

\begin{corollary}\label{cor:edge-disconnection}
Let $N=\binom{n}{2}$ and let $M\in
\{\floor{N/2}, \ceil{N/2}\}$. If $G_{n,M}$ is chosen uniformly from all graphs on $n$ vertices with $M$ edges, then $$\Pr[G_{n,M}\text{ is disconnected}] =o(1/n). $$
\end{corollary}

\begin{proof}
    Let $G$ be the random graph on $n$ vertices where every edge occurs independently with probability $1/2$. Conditional on $|E(G)|=M$, the graph $G$ is uniformly distributed over all graphs on $n$ vertices with exactly $M$ edges. It follows that
    \begin{align*}
        \Pr[G_{n,M}\text{ is disconnected}]
        &=
        \Pr[G\text{ is disconnected}\mid |E(G)| = M]\\
        &\le
        \frac{\Pr[G\text{ is disconnected}]}
        {\Pr[|E(G)| = M]}\\
        &=\frac{\Pr[G\text{ is disconnected}]}
        {2^{-N}\binom{N}{M}}
    \end{align*}
    From \cref{lem:gilbert-connectivity}, we have $\Pr[G \text{ is disconnected}] = O(n2^{-n})$. Applying Stirling's approximation to the central binomial coefficient, we have 
    $$
        2^{-N}\binom{N}{M} = \Theta(N^{-1/2})  = \Theta(1/n)$$ since $|M - N/2| = O(\sqrt{N})$. Therefore $$ \Pr[G_{n,M}\text{ is disconnected}] = \frac{O(n 2^{-n})}{\Theta(1/n)} = O(n^2 2^{-n}) = o(1/n). \qedhere
    $$
\end{proof}

\begin{theorem}\label{thm:expected}
    Let $\varepsilon$ be a uniformly random equitable $2$-coloring of $K_n$. Then $\E[\wld_\varepsilon(K_n)] \to 0$ as $n \to \infty$.
\end{theorem}

\begin{proof}
Let $N=\binom{n}{2}$. Since $\varepsilon$ is equitable, for $i=1,2$, we have $|E(G_i)| \in \{\floor{N/2}, \ceil{N/2}\}$. From \eqref{eq:two-colors}, the wild number of $K_n$ with $\varepsilon$ is  $$\wld(K_n) = (\kappa(G_1)-1)+ (\kappa(G_2) - 1).$$ The expected wild number of $K_n$ is then \[\E[\wld(K_n)] = \E[\kappa(G_1)-1] + \E[\kappa(G_2)-1].\] Conditional on its number of edges, each $G_i$ is uniformly distributed over all graphs on $n$ vertices with that many edges. Since a graph on $n$ vertices can have at most $n$ components, for $i=1,2$, we then have \[\E[\kappa(G_i)-1] \le (n-1)\Pr[G_i\text{ is disconnected}] =o(1)
    \] by \cref{cor:edge-disconnection}. Thus, 
    $$
        \E[\wld(K_n)] = \E[\kappa(G_1)-1] + \E[\kappa(G_2)-1] = o(1) \longrightarrow 0.\qedhere
    $$
\end{proof}

\subsection{Minimum Proper Colorings}

When $n$ is odd, $\chi'(K_n)=n$ and every color class is a near-perfect matching of size $(n-1)/2$; when $n$ is even, $\chi'(K_n)=n-1$ and every color class is a perfect matching of size $n/2$. For a color $c_i$, we write $M_i$ for its monochromatic matching. For odd $n$, we denote the isolated vertex in color $c_i$ by $v_i$. Note that each color misses exactly one vertex, and each vertex misses exactly one color, so $c_i \mapsto v_i$ is a bijection between colors and vertices; in particular, the $v_i$ are distinct.

For minimum proper colorings of $K_n$, ideal wild sets satisfy several structural constraints that will be used throughout this section.

\begin{lemma}\label{lem:proper-structure}
    Let $n\ge3$ and let $\rho$ be a minimum proper coloring of $K_n$. Every color-connecting set is incident to at least $n-1$ vertices, and to all $n$ vertices when $n$ is odd. Moreover, every ideal wild set is a forest in which every nontrivial component has at least three vertices.
\end{lemma}

\begin{proof}
    First, any color-connecting set must be incident to at least $n-1$ vertices. Indeed, if two vertices are not incident to any wild edge, then the unique edge joining them is isolated in its color class, preventing that color class from becoming connected. When $n$ is odd, each color class is a near-perfect matching with a unique isolated vertex. Consequently, every vertex must be incident to a wild edge, as an isolated vertex would remain isolated in its missing color. 
    
    Minimality of an ideal wild set further implies that every ideal wild set is a forest. If an edge $e$ lies on a cycle of a color-connecting set $W$, then the endpoints of $e$ remain joined by a path in $W-e$, thus $G_i^{W-e}$ would remain connected for all $i$ implying $W-e$ is a color-connecting set. By the same argument as above, no component formed by the edges of an ideal wild set can be a single edge as the edge would remain isolated in its own color class. Thus, every nontrivial component contains at least three vertices.
\end{proof}

\subsubsection{Lower Bound}
The properties of ideal wild sets discussed above provide a natural lower bound on the wild  number.

\begin{theorem}\label{thm:proper-kn-lowerbound}
For any minimum proper coloring $\rho$ of $K_n$ with $n \geq 3$, the wild number satisfies
$$
    \wld_\rho(K_n) \ge \floor{ \frac{4\ceil{n/2}}{3}}.
$$
\end{theorem}

\begin{proof}
    Let $W$ be an ideal wild set. Let $s$ be the number of vertices incident to an edge in $W$. By \cref{lem:proper-structure}, we have $s \ge n-1$ if $n$ is even and $s=n$ if $n$ is odd. In all cases,
    $s\ge 2\ceil{n/2}-1$. Further, we have that $W$ is a forest on $s$ vertices. Let $t$ be the number of tree components of $W$. Since each component contains at least $3$ vertices, by \cref{lem:proper-structure}, we have $t\le \floor{s/3}$ so $|W|=s-t\ge s-\floor{s/3}$. Note that the function $f(x)=x-\floor{x/3}$ is non-decreasing on the integers. Since $s\ge 2\ceil{n/2}-1$, we have
    $$
        |W|
        \ge
        \bigl(2\ceil{n/2}-1\bigr)
        -
        \left\lfloor
            \frac{2\ceil{n/2}-1}{3}
        \right\rfloor.
    $$ Setting $k=\ceil{n/2}$ and writing $k=3q+r$ with $r\in\{0,1,2\}$, both $2k-1-\floor{(2k-1)/3}$ and $\floor{4k/3}$ equal $4q+r$. Therefore, $|W|\ge \left\lfloor\frac{4\lceil n/2\rceil}{3}\right\rfloor$ and so
    $$
        \wld_\rho(K_n)\ge
        \left\lfloor\frac{4\lceil n/2\rceil}{3}\right\rfloor.\qedhere
    $$
\end{proof}

\subsubsection{Structure of Ideal Wild Sets} 

Recall that $v_i$ is the vertex isolated in color $c_i$. We call a set $T=\{v_i,v_j,v_k\}\subseteq V(K_n)$ a \emph{valid triple} if $\rho(v_jv_k)\neq c_i$, $\rho(v_iv_k)\neq c_j$, and $\rho(v_iv_j)\neq c_k$. This condition is necessary for $T$ to occur as a component of a color-connecting set, since otherwise there would be no edge of color $c_i$, $c_j$, or $c_k$ connecting $T$ to the remaining vertices. We call $T$ \emph{invalid for $c_i$} if $T$ is not a valid triple because $\rho(v_jv_k)=c_i$, or simply \emph{invalid} when color does not matter.

\begin{lemma} \label{lem:valid-triple-existence}
Let $\rho$ be a minimum proper coloring of $K_n$. For odd $n\ge7$, the graph $K_n$ contains a valid triple.
\end{lemma}

\begin{proof}
Each edge $xy\in M_i$ witnesses the invalid triple
$\{v_i,x,y\}$. Conversely, every invalid triple is witnessed by at least one of the edges between its vertices.

An edge witnesses exactly one invalid triple, since its color $c_i$ determines a unique vertex $v_i$. Therefore, the number of invalid triples is at most
$|E(K_n)|=\binom{n}{2}$. Since $n\ge7$, the number of total triples, $\binom{n}{3}$, exceeds $\binom{n}{2}$.
Hence not every triple is invalid, and a valid triple exists.
\end{proof}

When $n$ is odd, \cref{lem:proper-structure} implies that any ideal wild set must be a spanning forest where every component contains at least $3$ vertices. We call the vertex sets of these components \emph{blocks}. 

A forest of blocks fails to be color-connecting when a nonempty union of blocks is closed under some color class and is, therefore, disconnected from the remaining blocks. We isolate this obstruction in the following definition.

\begin{definition}\label{def:seal} Let $\rho$ be a minimum proper coloring of $K_n$ with $n$ odd, and let $P = \{B_1, \dots, B_t\}$ be a partition of $V(K_n)$ into blocks where $|B_j| \ge 3$ for $1 \le j \le t$. A \emph{seal} is a pair $(c_i, \mathcal{U})$, where $c_i$ is a color and $\mathcal{U} \subseteq P$ is a nonempty set of blocks whose vertex union $U = \bigcup_{B \in \mathcal{U}} B$ satisfies:
\begin{enumerate}[label=(\roman*)]
    \item $v_i \notin U$, and
    \item\label{it:seal-closed} every edge of $M_i$ with an endpoint in $U$ has both endpoints in $U$; equivalently, $M_i$ restricted to $U$ is a perfect matching of $U$.
\end{enumerate}
\end{definition}

We call $P$ \emph{seal-free} if it admits no seal. Since a perfect matching requires an even number of vertices, $|U|$ is even for every seal; in particular, a single odd block cannot by itself form a seal. Up to relabeling, the two possible seals consisting of two triples are shown in \cref{fig:order2-seals}.

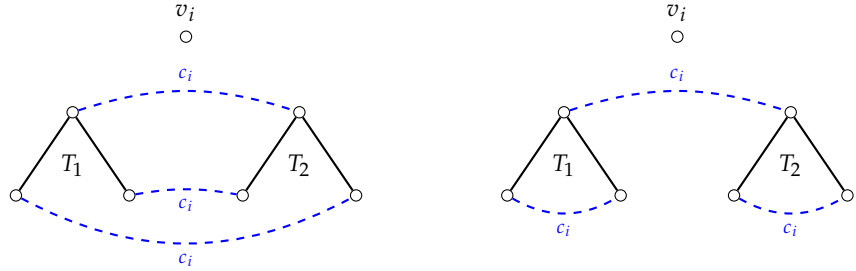
\begin{figure}[ht]
    \centering
    \begin{tikzpicture}
    \begin{scope}[xshift=0cm]
      \node[vertex] (a1) at (0,0) {};
      \node[vertex] (a2) at (0.75,1.1) {};
      \node[vertex] (a3) at (1.5,0) {};
      \node at (0.75,0.4) {\small $T_1$};
      \node[vertex] (b1) at (3,0) {};
      \node[vertex] (b2) at (3.75,1.1) {};
      \node[vertex] (b3) at (4.5,0) {};
      \node at (3.75,0.4) {\small $T_2$};
      \node[vertex,label=above:{\small $v_i$}] (v) at (2.25,2.1) {};
      
      \draw[wild edge] (a1)--(a2)--(a3);
      \draw[wild edge] (b1)--(b2)--(b3);
      
      \draw[cedge] (a3) to[bend left=10] node[midway,below]{\scriptsize $c_i$} (b1);
      \draw[cedge] (a2) to[bend left=18] node[midway,above]{\scriptsize $c_i$} (b2);
      \draw[cedge] (a1) to[bend right=28] node[midway,below]{\scriptsize $c_i$} (b3);
    \end{scope}

    \begin{scope}[xshift=6.5cm]
      \node[vertex] (c1) at (0,0) {};
      \node[vertex] (c2) at (0.75,1.1) {};
      \node[vertex] (c3) at (1.5,0) {};
      \node at (0.75,0.4) {\small $T_1$};
      \node[vertex] (d1) at (3,0) {};
      \node[vertex] (d2) at (3.75,1.1) {};
      \node[vertex] (d3) at (4.5,0) {};
      \node at (3.75,0.4) {\small $T_2$};
      \node[vertex,label=above:{\small $v_i$}] (w) at (2.25,2.1) {};
      \draw[wild edge] (c1)--(c2)--(c3);
      \draw[wild edge] (d1)--(d2)--(d3);
      
      \draw[cedge] (c1) to[bend right=30] node[midway,below]{\scriptsize $c_i$} (c3);
      \draw[cedge] (d1) to[bend right=30] node[midway,below]{\scriptsize $c_i$} (d3);
      \draw[cedge] (c2) to[bend left=18] node[midway,above]{\scriptsize $c_i$} (d2);
    \end{scope}
\end{tikzpicture}
    \caption{The two seals, up to relabeling, of two triples for a color $c_i$.}
    \label{fig:order2-seals}
\end{figure}

\begin{proposition}\label{prop:seal}
Let $\rho$ be a minimum proper coloring of $K_n$ with $n$ odd and let $P  = \{B_1, \dots, B_t\}$ be a partition of $V(K_n)$ into blocks where $|B_j| \ge 3$ for $1 \le j \le t$. For each $j$, let $T_j$ be a spanning tree of $K_n[B_j]$, and let $W = \cup_{j=1}^t E(T_j)$. Then $W$ is color-connecting if and only if $P$ is seal-free.
\end{proposition}

\begin{proof}
First, suppose $P$ has a seal $(c_i, \mathcal{U})$, with $U = \bigcup_{B \in \mathcal{U}} B$ and $\overline{U} = V(K_n) - U$. We claim no edge of $G_i^W$, whose edge set is $M_i \cup W$, joins $U$ to its complement $\overline{U}$. The blocks of a partition are pairwise disjoint, so every block lies entirely in $U$ or entirely in $\overline{U}$. A wild edge $e \in W$ has both endpoints in a single block, and that block lies on one side, so no wild edge crosses the cut $(U, \overline{U})$; by condition~\ref{it:seal-closed} of \cref{def:seal}, no edge of $M_i$ crosses either. Both sides are nonempty; $U \neq \emptyset$ since $\mathcal{U} \neq \emptyset$ and $v_i \in \overline{U}$ since $v_i \notin U$. So $G_i^W$ is disconnected and $W$ is not color-connecting.

Conversely, suppose $W$ is not color-connecting, say $G_i^W$ is disconnected for the color $c_i$. Each block of $P$ is connected by its wild spanning tree, so every component of $G_i^W$ is a union of whole blocks. Hence the blocks of $P$ split into two nonempty sets $\mathcal{A}$ and $\mathcal{B}$ with no edge of $M_i$ between $A = \bigcup \mathcal{A}$ and $B = \bigcup \mathcal{B}$. Note that $M_i$ is a near-perfect matching missing exactly the vertex $v_i$, which lies in either $A$ or $B$, say $v_i \in A$. Then every vertex of $B$ is matched by $M_i$, and since no edge colored $c_i$ has exactly one endpoint in $B$, it has both endpoints in $B$. Thus $M_i$ restricted to $B$ is a perfect matching, and $(c_i, \mathcal{B})$ is a seal. Note that neither direction referred to the choice of spanning trees inside the blocks, so color-connectivity depends only on the partition.
\end{proof}

\begin{lemma}\label{lem:k4-rainbow}
Let $\rho$ be a minimum proper coloring of $K_n$ with $n \ge 7$ odd, and let $A \subseteq V(K_n)$ where $|A| = 4$. If $A$ is a block of a partition $P$, then $(c_i,\{A\})$ is not a seal for any $i$ if and only if $K_n[A]$ is a rainbow $K_4$.
\end{lemma}

\begin{proof}
    Since $\rho$ is proper, if some color $c_i$ appears twice in $K_n[A]$, the two edges colored $c_i$ are disjoint and cover $A$; hence $v_i\notin A$ and $(c_i, \{A\})$ is a seal. Conversely, since $|A|=4$, a seal $(c_i, \{A\})$ is a pair of disjoint edges colored $c_i$ between the vertices of $A$, so $c_i$ is a repeated color. Thus, a seal $(c_i, \{A\})$ exists for some $c_i$ if and only if $K_n[A]$ is not rainbow.
\end{proof}

\begin{lemma}\label{lem:complement-collapse}
    Let $\rho$ be a minimum proper coloring of $K_n$ with $n \ge 7$ odd, and let $P$ be a partition of $V(K_n)$ into blocks, one of which is a triple $T$. Let $\mathcal{U} = P \setminus \{T\}$. Then $(c_i, \mathcal{U})$ is a seal if and only if $T$ is invalid for $c_i$.
\end{lemma}

\begin{proof}
Suppose $(c_i, \mathcal{U})$ is a seal. By definition, $v_i \notin U$, so $v_i \in T$. Since $|T| = 3$, the remaining set $U = V(K_n)\setminus T$ contains $n-3$ vertices. Since $M_i$ restricted to $U$ is a perfect matching of $U$, it uses $(n-3)/2$ edges of $M_i$, so exactly $(n-1)/2-(n-3)/2=1$ edge of $M_i$ remains. Since every vertex of $U$ is already matched, this edge has no endpoint in $U$. Thus it must have both endpoints in $T \setminus \{v_i\}$. It follows that $T$ contains the isolated vertex $v_i$ and its other two vertices are joined by an edge colored $c_i$, making $T$ invalid for $c_i$. 

Conversely, if $T$ is invalid for $c_i$, then $T = \{v_i, a, b\}$ with $\rho(ab) = c_i$. The remaining $(n-1)/2 - 1 = (n-3)/2$ edges of $M_i$ cover the remaining $n-3$ vertices of the graph, which are $V(K_n) \setminus T = U$. Those edges form a perfect matching of $U$. Thus, $(c_i, \mathcal{U})$ is a seal.
\end{proof}

\subsubsection{Small Cases}\label{sec:small-cases}
The following lemma allows us to bound the wild number of an even complete graph by an odd complete graph.

\begin{lemma}\label{lem:odd-even-transfer}
Let $k \geq 2$, let $\rho$ be a minimum proper coloring of $K_{2k}$, and let $v \in V(K_{2k})$. Define 
$$\rho_v:=\rho\rest{K_{2k}-v}.$$ Then $\rho_v$ is a minimum proper coloring of $K_{2k-1}$ and 
$$\wld_\rho(K_{2k}) \le \wld_{\rho_v}(K_{2k-1}).$$
\end{lemma}

\begin{proof}
    The coloring $\rho$ uses $2k-1$ colors, and each color class is a perfect matching of size $k$. Deleting $v$ removes exactly one edge from each color class. Thus every color class of $\rho_v$ is a near-perfect matching of size $k-1$ and $\rho_v$ is a proper coloring of $K_{2k-1}$ using $2k-1$ colors, and so is minimum as $\chi'(K_{2k-1})=2k-1$.
    
    Let $W\subseteq E(K_{2k}-v)$ be color-connecting under $\rho_v$. Fix a color $c_i$ and let $e_i$ be the unique edge of $M_i$ incident to $v$. The graph $(M_i-e_i)\cup W$ is connected on $V(K_{2k}-v)$ and $e_i$ joins $v$ to a vertex of $(M_i-e_i)\cup W$. Hence $M_i\cup W$ is connected on $V(K_{2k})$.
    
    This holds for every color $c_i$, so $W$ is also color-connecting in $K_{2k}$. Therefore,
    $$
        \wld_\rho(K_{2k}) \le \wld_{\rho_v}(K_{2k-1}).\qedhere
    $$
\end{proof}

\begin{lemma}\label{lem:k7-sufficient}
    Let $\rho$ be a minimum proper coloring of $K_7$. Suppose that $T$ is a valid triple and that $B$ induces a rainbow $K_4$ and is disjoint from $T$. Then the union of the spanning trees of $K_7[T]$ and $K_7[B]$ is a color-connecting set of size $5$.
\end{lemma}

\begin{proof}
    Note that $P= \{T,B\}$ is a block partition of $V(K_7)$ and $W$, consisting of a spanning tree of each block, has $|W| = (3-1) +(4-1) = 5$. By \cref{prop:seal}, $W$ is color-connecting if the partition $P$ admits no seals. Recall that a seal perfectly matches its vertices, requiring an even number of vertices. The only nonempty subset of $P$ whose union has even size is $\{B\}$. By \cref{lem:k4-rainbow}, $(c_i,\{B\})$ is a seal for some $c_i$ if and only if $K_7[B]$ is not rainbow. From our assumptions, we have that $K_7[B]$ is rainbow, so no seal has the form $(c_i,\{B\})$. Thus, $W$ is color-connecting.
\end{proof}

\begin{lemma}\label{lem:k9-existence}
    Let $\rho$ be a minimum proper coloring of $K_9$. Then we can partition $V(K_9)$ into three valid triples.
\end{lemma}

\begin{proof}
    By \cref{lem:valid-triple-existence}, there is a valid triple; relabel so that it is $T=\{v_1,v_2,v_3\}$. Each edge $v_jv_k$ of $T$ is not colored $c_j$ or $c_k$, since $v_j$ and $v_k$ are isolated in those colors, and is not colored $c_i$, since $T$ is valid. As $\rho$ is proper, the three edges of $T$ receive distinct colors, which we relabel $c_4,c_5,c_6$. It remains to partition the remaining six vertices into two valid triples. There are $\binom{6}{3}/2=10$ ways to partition the remaining six vertices into two unordered triples.
    
    We now bound the number of partitions that do not produce the remaining two valid triples. Since $v_1,v_2,v_3\in T$, no invalid triple among the remaining six vertices can arise from colors $c_1,c_2,c_3$. For $\{c_4,c_5,c_6\}$, one matching edge lies inside $T$, and the third vertex of $T$ is matched outside $T$, leaving two matching edges among the remaining six vertices. Thus, at most $2$ remaining edges of each color can contribute invalid triples. For $\{c_7,c_8,c_9\}$, no matching edge lies inside $T$, so the three vertices of $T$ are matched to three vertices outside $T$, leaving one matching edge among the remaining six vertices. Thus, at most $1$ remaining edge of each color can contribute invalid triples. Therefore, among the remaining six vertices, there are at most $3\cdot 2+3\cdot 1=9$ invalid triples.
    
    Each invalid triple determines at most one partition of the remaining six vertices that does not produce the two remaining valid triples, namely the partition consisting of that triple and its complement. Therefore, at most $9$ of the $10$ partitions are invalid. Hence at least one partition of the remaining six vertices consists of two valid triples. $T$ together with these two valid triples gives a partition of $V(K_9)$ into three valid triples.
\end{proof}

\begin{lemma}\label{lem:validity-sufficient}
    Let $\rho$ be a minimum proper coloring of $K_9$. Suppose that $T_1,T_2,T_3\subseteq V(K_9)$ are valid triples such that $T_1, T_2, T_3$ are pairwise disjoint. Then the union of the spanning trees of $K_9[T_1], K_9[T_2]$, and $K_9[T_3]$ is a color-connecting set of size $6$.
\end{lemma}

\begin{proof}
    Note that $P = \{T_1, T_2, T_3\}$ is a block partition of $V(K_9)$, and by \cref{prop:seal}, it suffices to rule out the existence of seals. The union of blocks in a seal contains an even number of vertices, so $\mathcal{U}$ must consist of exactly two triples. By \cref{lem:complement-collapse}, a seal $(c_i, P\setminus\{T_j\})$ exists exactly when $T_j$ is invalid for $c_i$. Let $W$ be the union of the spanning trees of $K_9[T_1], K_9[T_2]$, and $K_9[T_3]$. All three triples are valid, so $P$ is seal-free and $W$ is color-connecting with $|W| = 3(3-1)=6$.
\end{proof}

\begin{lemma}\label{lem:k11-existence}
    Let $\rho$ be a minimum proper coloring of $K_{11}$. Then we can partition $V(K_{11})$ into a valid triple $T$ and two blocks that each induce a rainbow $K_4$.
\end{lemma}

\begin{proof}
    By \cref{lem:valid-triple-existence}, there is a valid triple; relabel so that it is $T=\{v_1,v_2,v_3\}$. Each edge $v_jv_k$ of $T$ is not colored $c_j$ or $c_k$, since $v_j$ and $v_k$ are isolated in those colors, and is not colored $c_i$, since $T$ is valid. As $\rho$ is proper, the three edges of $T$ receive distinct colors, which we relabel $c_4,c_5,c_6$. It remains to partition the remaining eight vertices into two sets, each inducing a rainbow $K_4$. There are $\binom{8}{4}/2=35$ partitions of these eight vertices into two sets of size $4$.

    For the colors $\{c_1,c_2,c_3\}$, one vertex of $T$ is missed, while the other two vertices of $T$ are matched to vertices outside $T$, since $T$ is valid. Thus, two edges of each such color are incident to $T$, leaving \(5-2=3\) edges among the remaining eight vertices. For $\{c_4,c_5,c_6\}$, one edge of each color lies inside $T$, while the remaining vertex of $T$ is matched to a vertex outside $T$. Thus, again, two edges of each such color are incident to $T$, leaving $5-2=3$ edges among the remaining eight vertices. For $\{c_7,\dots,c_{11}\}$, no edge lies inside $T$ and no vertex of $T$ is missed, so all three vertices of $T$ are matched to vertices outside $T$. Thus, there are $5-3=2$ edges of each such color among the remaining eight vertices.

    We now bound the number of partitions that do not produce two disjoint rainbow $K_4$. For $\{c_1, \dots, c_6\}$, there are $3$ matching edges among the remaining eight vertices. A $4$-set on which color $c_i$ appears twice is the vertex set of a pair of these edges, so each such color yields at most $\binom{3}{2}=3$ non-rainbow $4$-sets, for at most $6 \cdot 3=18$ in total. For $\{c_7, \dots, c_{11}\}$, there are exactly two matching edges among the remaining eight vertices, and these determine a unique non-rainbow $K_4$. Thus, these five colors produce at most $5$ additional $K_4$ that are not rainbow. Each $K_4$ that is not rainbow determines a unique unordered partition of the remaining eight vertices into two sets of size $4$. Since the total number of bad partitions is at most $18+5=23<35$, we have at least one partition into two disjoint rainbow $K_4$.
\end{proof}

\begin{lemma}\label{lem:k11-sufficient}
    Let $\rho$ be a minimum proper coloring of $K_{11}$. Suppose that $T \subseteq V(K_{11})$ is a valid triple and that $B_1,B_2 \subseteq V(K_{11})$ each induce a rainbow $K_4$ such that $T, B_1, B_2$ are pairwise disjoint. Then the union of the spanning trees of $K_{11}[T], K_{11}[B_1]$, and $K_{11}[B_2]$ is a color-connecting set of size $8$. 
\end{lemma}

\begin{proof}
     Note that $P = \{T, B_1, B_2\}$ is a block partition of $V(K_{11})$, and by \cref{prop:seal}, it suffices to rule out the existence of seals. Recall that a seal requires an even number of vertices, so we need only consider the sets $\{B_1\}$, $\{B_2\}$, and $\{B_1, B_2\}$. By \cref{lem:k4-rainbow}, a seal on $\{B_1\}$ or $\{B_2\}$ exists exactly when the corresponding $K_4$ is not rainbow. By \cref{lem:complement-collapse}, a seal on $\{B_1, B_2\} = P \setminus \{T\}$ exists exactly when $T$ is invalid. So $P$ is seal-free when $T$ is valid and $K_{11}[B_1], K_{11}[B_2]$ are rainbow. Let $W$ be the union of the spanning trees of $K_{11}[T], K_{11}[B_1]$, and $K_{11}[B_2]$, so $|W|=(3-1)+(4-1)+(4-1)=8$. By \cref{prop:seal}, it follows that $W$ is color-connecting when $T$ is valid and $K_{11}[B_1], K_{11}[B_2]$ are rainbow. 
\end{proof}

\begin{theorem} \label{thm:small-complete-prop-nums}
    Let $3 \le n \le 12$ and let $\rho $ be a minimum proper coloring of $K_{n}$. Then 
    $$ \wld_\rho(K_{n})=\left\lfloor\frac{4\lceil n/2\rceil}{3}\right\rfloor.
    $$
\end{theorem}

\begin{proof}
    Given a minimum proper coloring $\rho$ of $K_n$, \cref{thm:proper-kn-lowerbound} establishes that $$\wld_\rho(K_n) \ge \left\lfloor\frac{4\lceil n/2\rceil}{3}\right\rfloor.$$ We show that this lower bound is sharp for $3 \le n \le 12$. We first consider odd $n$ with $3 \le n \le 11$. 
    
    For $n \in \{3,5\}$, the general bound $\wld_\rho(K_n)\le n-1$ from \cref{eq:def-bound} gives $$ \wld_\rho(K_3) \le 2 \text{ and } \wld_\rho(K_5) \le 4, $$ which equal the respective lower bounds.

    For $n=7$, by \cref{lem:valid-triple-existence}, there exists a valid triple $T\subseteq V(K_7)$. Let
    $B = V(K_7)\setminus T$, so $|B|=4$. By \cref{lem:complement-collapse}, a seal of the form $(c_i,\{B\})$ exists if and only if $T$ is invalid. Since $T$ is valid, the partition $P=\{T,B\}$ is seal-free. Equivalently, by \cref{lem:k4-rainbow}, $K_7[B]$ is rainbow. Choosing a spanning tree in each block forms a color-connecting set of size $(3-1)+(4-1)=5$ by \cref{lem:k7-sufficient}, so 
    $$\wld_\rho(K_7) \le 5 .$$

    For $n=9$, by \cref{lem:k9-existence}, there exists a partition of $V(K_9)$ into three disjoint valid triples $T_1, T_2, T_3$. By \cref{lem:validity-sufficient}, the union of the spanning trees of $K_9[T_1],K_9[T_2]$, and $K_9[T_3]$ forms a color-connecting set $W$ of size $|W| = 3 \cdot 2 = 6$. Thus, $$\wld_\rho(K_9) \le 6.$$

    For $n=11$, by \cref{lem:k11-existence}, we have a partition of $V(K_{11})$ into $T,B_1,B_2$ where $T$ is a valid triple and $B_1,B_2$ induce rainbow $K_4$. By \cref{lem:k11-sufficient}, $T,B_1,B_2$ form the components of a color-connecting set $W$ for $K_{11}$ with $|W|=8$. Thus, $$\wld_\rho(K_{11}) \le 8.$$ These inequalities agree with the respective lower bounds, completing the proof for odd $n$ with $3 \le n \le 11$.
    
    We now consider even $n$ with $4 \le n \le 12$ . Choose $v\in V(K_n)$, then, by \cref{lem:odd-even-transfer}, the restriction $\rho_v$ is a minimum proper coloring of $K_{n-1}$ and $$\wld_\rho(K_n) \le \wld_{\rho_v}(K_{n-1}) =
    \left\lfloor \frac{4\lceil(n- 1)/2\rceil}{3} \right\rfloor = \left\lfloor \frac{4\lceil n/2\rceil}{3} \right\rfloor, $$ where the last equality holds since $n$ is even. By the odd cases established above, the lower bound completes the proof for even $n$ with $4 \le n \le 12$.
\end{proof}

\subsubsection{Weak Upper Bound}

\begin{lemma}\label{lem:crossing-edges-triple}
Let $\rho$ be a minimum proper coloring of $K_n$ with $n$ odd and let $T \subseteq V(K_n)$ with $|T| = 3$.
    \begin{enumerate}[label=(\roman*)]
        \item\label{it:triple-i} If $v_i \notin T$, then $K_n$ has an edge of color $c_i$ with exactly one endpoint in $T$.
    
        \item\label{it:triple-ii} If $T = \{v_i, a, b\}$, then $K_n$ has an edge of color $c_i$ with exactly one endpoint in $T$ if and only if $\rho(ab) \neq c_i$.
    \end{enumerate}
\end{lemma}

\begin{proof}
Since $n$ is odd, $M_i$ is a near-perfect matching covering every vertex except $v_i$. We seek an edge colored $c_i$ with exactly one endpoint in $T$.

To prove~(i), suppose $v_i \notin T$, so all three vertices of $T$ are matched by $M_i$. Since a set of three vertices can contain at most one edge of the matching $M_i$, at least one vertex of $T$ is matched by $M_i$ to a vertex outside $T$.

To prove~(ii), suppose $T = \{v_i, a, b\}$. The vertex $v_i$ is isolated in $M_i$, while each of $a$ and $b$ is incident to exactly one edge of $M_i$. These edges have both of their endpoints in $T$ if and only if $a$ and $b$ are
adjacent to each other, that is, $ab \in M_i$, or $\rho(ab) = c_i$. If $\rho(ab) \neq c_i$, at least one of $a, b$ shares an edge colored $c_i$ with a vertex outside of $T$. If instead $\rho(ab) = c_i$, then $ab$ is the only edge colored $c_i$ incident to $T$, so there is no edge colored $c_i$ with exactly one endpoint in $T$.
\end{proof}

\begin{theorem} \label{thm:proper-odd-upper}
    Let $n\ge7$ be odd and let $\rho$ be a minimum proper coloring of $K_n$. Then
    $$\wld_\rho(K_n)\le n-2.$$
\end{theorem}

\begin{proof}
    By \cref{lem:valid-triple-existence}, there exists a valid triple $T\subseteq V(K_n)$. Let
    $B=V(K_n)\setminus T$.
    Choose a spanning tree of $K_n[T]$ and a spanning tree of $K_n[B]$, and let $W$ be the union of their edges. Then
    $$|W| = (|T|-1)+(|B|-1) = 2+(n-4) = n-2.$$
    
    Fix a color $c_i$. If $v_i\notin T$, then \cref{lem:crossing-edges-triple}\ref{it:triple-i} gives an edge of color $c_i$ between the vertices of $T$ and $B$. If $v_i\in T$, then the fact that $T$ is a valid triple together with \cref{lem:crossing-edges-triple}\ref{it:triple-ii} gives an edge of color $c_i$ between the vertices of $T$ and $B$.
    
    The wild edges make $T$ and $B$ each connected, while the edge of color $c_i$ guaranteed by \cref{lem:crossing-edges-triple} joins the two wild components. Therefore, $G_i^W$ is connected. Since $c_i$ was arbitrary, $W$ is color-connecting and we have 
    $$
        \wld_\rho(K_n)\le n-2.\qedhere
    $$
\end{proof}

\begin{corollary}\label{cor:proper-even-upper}
    Let $n\ge8$ be even, and let $\rho$ be a minimum proper coloring of $K_n$. Then
    $$
        \wld_\rho(K_n)\le n-3.
    $$
\end{corollary}

\begin{proof}
    Let $v \in V(K_n)$. Combining \cref{lem:odd-even-transfer} with \cref{thm:proper-odd-upper}, since $n-1\ge7$ is odd, we have 
    $$
        \wld_\rho(K_n)
        \le\wld_{\rho_v}(K_{n-1}) \le (n-1)-2 = n-3.\qedhere
    $$
\end{proof}

\subsubsection{Performance of the Greedy Algorithm}

We conclude our study of minimum proper colorings of $K_n$ considering Algorithm~5.10 of \cite{WildNumber}, an algorithm that produces a color-connecting set. Given a set $W$ of edges already colored wild, the \emph{dip number} of an edge $e$ is $$\delta_W(e) = \sum_{i=1}^\ell \bigl(\kappa(G_i^W)-\kappa(G_i^{W \cup \{e\}})\bigr),$$ the number of colors $c_i$ for which coloring $e$ wild merges two components of $G_i^W$. The greedy algorithm repeatedly selects an edge of maximum dip number until the selected set is color-connecting, breaking ties with the potential rule of \cite{WildNumber}; our arguments do not depend on how ties are broken.

While the output of Algorithm~5.10 is not necessarily ideal, the color-connecting set produced is always a forest.

\begin{lemma}\label{lem:greedy-acyclic}
    Let $\rho$ be a minimum proper coloring of $K_n$, and let
    $\Wg$ be the output of Algorithm~5.10 of \cite{WildNumber}. Then $\Wg$ is acyclic. Consequently, $|\Wg|\le n-1$.
\end{lemma}

\begin{proof}
    Suppose the current selected set of edges is $W$ and the algorithm has not yet terminated. Then there is a color $c_i$ for which $G_i^W$ is disconnected. Since $K_n$ is complete, there is an edge $uv$ joining two distinct components of $G_i^W$.
    
    The edge $uv$ cannot have color $c_i$, otherwise it would already belong to $M_i$, contradicting the assumption that its endpoints lie
    in distinct components of $G_i^W$. Thus, coloring $uv$ wild merges two components of $G_i^W$, so $uv$ has a positive dip number.
    
    Therefore, whenever the algorithm has not terminated, an edge of positive dip number exists. Since the algorithm selects an edge of maximum dip number, it never selects an edge of dip number zero.
    
    We claim that an edge of positive dip number joins two distinct components of the current selected wild set $W$. Suppose instead that the endpoints $u$ and $v$ of an edge were already connected by a path in $W$. Since $W\subseteq E(G_i^W)$ for every color $c_i$, the vertices $u$ and $v$ would already be in the same component of $G_i^W$ for every color. Adding $uv$ would therefore not merge monochromatic components, so its dip number would be zero.
    
    Thus, every edge selected by the algorithm joins two distinct components of the current selected set. Adding such an edge cannot create a cycle. Since the algorithm begins with an empty set, $\Wg$ is acyclic.
    
    As a forest on at most $n$ vertices, it has at most $n-1$ edges. Therefore, $|\Wg|\le n-1$.
\end{proof}

\begin{lemma}\label{lem:color-saturation}
    Let $W$ be a color-connecting set for $K_n$. If $M_i\subseteq W$ for some color $c_i$, then the edges of $W$ form a connected spanning subgraph of $K_n$. Consequently, $|W|\ge n-1$.
\end{lemma}

\begin{proof}
    Since $W$ is color-connecting, the spanning subgraph $G_i^W$ is connected. By hypothesis, $M_i \subseteq W$, so $E(G_i^W)=M_i \cup W=W$. Every connected graph on $n$ vertices has at least $n-1$ edges, so $|W|\ge n-1$.
\end{proof}

\begin{theorem}
\label{thm:greedy-bounds}
    Let $\rho$ be a minimum proper coloring of $K_n$ and let
    $\Wg$ be the output of Algorithm~5.10 of
    \cite{WildNumber}. For $n \ge 3$, we have
    $$
        \frac{|\Wg|}{\wld_\rho(K_n)}
        \le
        \frac{n-1}{
            \left\lfloor
                \frac{4\lceil n/2\rceil}{3}
            \right\rfloor
        },
    $$
    and
    $$
        0
        \le
        |\Wg|-\wld_\rho(K_n)
        \le
        n-1-
        \left\lfloor
            \frac{4\lceil n/2\rceil}{3}
        \right\rfloor.
    $$
    In particular, this upper bound on the approximation ratio tends to
    $3/2$ as $n$ tends to infinity.
\end{theorem}

\begin{proof}
    By \cref{lem:greedy-acyclic}, $|\Wg|\le n-1$. By \cref{thm:proper-kn-lowerbound},
    $
        \wld_\rho(K_n)
        \ge
        \left\lfloor
            \frac{4\lceil n/2\rceil}{3}
        \right\rfloor.
    $
    Combining these inequalities gives
    $$
        \frac{|\Wg|}{\wld_\rho(K_n)}
        \le
        \frac{n-1}{
            \left\lfloor
                \frac{4\lceil n/2\rceil}{3}
            \right\rfloor
        }.
$$

Since $\Wg$ is color-connecting and
$\wld_\rho(K_n)$ is the minimum size of a color-connecting set, we have $|\Wg|-\wld_\rho(K_n) \ge 0$. The preceding upper and lower bounds also give
$$|\Wg|-\wld_\rho(K_n) \le n-1- \left\lfloor \frac{4\lceil n/2\rceil}{3} \right\rfloor.$$

Finally, since $\floor{4\ceil{n/2}/3}
    =
    2n/3+O(1)$, as $n \rightarrow \infty$, we have 
    $$
        \frac{n-1}{
        \left\lfloor
            \frac{4\lceil n/2\rceil}{3}
        \right\rfloor
        } \longrightarrow
        \frac{3}{2}.\qedhere
    $$
\end{proof}

The following corollary identifies color-class saturation as a
structural failure of the greedy algorithm. When saturation
occurs, the algorithm returns a spanning tree, which is the largest
possible acyclic output, even though a smaller color-connecting set exists whenever $n\ge7$.

\begin{corollary}
\label{cor:greedy-saturation}
    Let $\rho$ be a minimum proper coloring of $K_n$, and suppose that Algorithm~5.10 of \cite{WildNumber} saturates a color
    class, that is,
    $M_i\subseteq \Wg$
    for some color $c_i$. Then
    $$
        |\Wg|=n-1.
    $$
    In that case, $\Wg$ is not ideal whenever $n\ge7$. More precisely,
    $$
        |\Wg|-\wld_\rho(K_n)
        \ge
        \begin{cases}
            1, & n\ge7\text{ odd},\\
            2, & n\ge8\text{ even}.
        \end{cases}
    $$
\end{corollary}

\begin{proof}
    By \cref{lem:color-saturation}, $|\Wg|\ge n-1$. By \cref{lem:greedy-acyclic}, $|\Wg|\le n-1$. Therefore, $|\Wg|=n-1$.
    
    Suppose that $n\ge7$ is odd. By
    \cref{thm:proper-odd-upper}, $\wld_\rho(K_n)\le n-2$.
    Thus,
    $$|\Wg|-\wld_\rho(K_n) \ge (n-1)-(n-2) = 1.
    $$
    
    Now suppose that $n\ge8$ is even. By
    \cref{cor:proper-even-upper}, $\wld_\rho(K_n)\le n-3$.
    Therefore,
    $$
        |\Wg|-\wld_\rho(K_n)
        \ge
        (n-1)-(n-3)
        =
        2.\qedhere
    $$
\end{proof}

\section{Complete Bipartite Graphs}\label{sec:complete-bipartite}

As in the previous section, we focus primarily on equitable $2$-colorings and minimum proper colorings. Throughout, $V_m$ and $V_n$ denote the parts of $K_{m,n}$, with $|V_m|=m$ and $|V_n|=n$, and we write $K_{x,0}$ and $K_{0,x}$ for a subset of $x$ isolated vertices of $V_m$ and $V_n$, respectively. First, we make an observation about subgraphs of complete bipartite graphs.

\begin{lemma}\label{lem:bip-max-edges}
    Let $G$ be a spanning subgraph of $K_{m,n}$, and define
    \[
        V_A=\{v\in V_m: d_G(v)\ge 1\} \quad\text{and}\quad
        V_B=\{v\in V_n: d_G(v)\ge 1\}.
    \]
    Let $A=|V_A|$ and $B=|V_B|$. Then $|E(G)|\le A(B-\kappa^*(G)+1)$.
\end{lemma}

\begin{proof}
    Let $(a_1,b_1), \dots, (a_{\kappa^*(G)},b_{\kappa^*(G)})$ be the bipartition sizes of each nontrivial component of $G$ where $a_i$ is the number of vertices belonging to $V_A$ and $b_i$ is the number of vertices belonging to $V_B$. Since the nontrivial components partition $V_A \cup V_B$, we have 
    $$
        \sum_{i=1}^{\kappa^*(G)} a_i=A\quad\text{ and }\quad \sum_{i=1}^{\kappa^*(G)} b_i=B.
    $$
    Note that $|E(G)|\le \sum_{i=1}^{\kappa^*(G)}a_ib_i$. Since every edge of $G$ has an endpoint in $V_n$, each nontrivial component of $G$ contains at least one vertex of $V_B$. For all $i\in [\kappa^*(G)]$, we then have that $b_i\le B-(\kappa^*(G)-1)=B-\kappa^*(G)+1$ and so 
    $$
        |E(G)|\le \sum_{i=1}^{\kappa^*(G)}a_ib_i\le (B-\kappa^*(G)+1)\sum_{i=1}^{\kappa^*(G)}a_i=A(B-\kappa^*(G)+1).\qedhere
    $$
\end{proof}

\subsection{Equitable $2$-Colorings}\label{sec:bip-eq2}
An equitable $2$-coloring of $K_{m,n}$ has two color classes containing $\floor{mn/2}$ and $\ceil{mn/2}$ edges. We obtain a sharp upper bound on $\wld(K_{m,n})$ for such a coloring by considering the possible structures of disconnected color classes. We first construct a family of equitable $2$-colorings $\varepsilon_{2^*}$ (see \cref{fig:kmn-eq2}).

\begin{construction}\label{con:eq2bip}
    Let $2 \le m \le n$. We construct an equitable $2$-coloring $\varepsilon_{2^*}$ such that every color class has $\floor{n/2}+1$ connected components.
    
    For even $n$, partition the vertices of $V_n$ into two sets of size $n/2$. Color all edges incident to the first set with $c_1$ and all edges incident to the second set with $c_2$. Each color class therefore induces a copy of $K_{m,n/2}$. In each color, the remaining $n/2$  vertices in $V_n$ are isolated while the vertices incident to edges of each color form a single nontrivial component. Thus, each color class has $n/2+1$ connected components.
    
    For odd $n$, partition the vertices of $V_n$ into two sets of size $\floor{n/2}$ and one remaining vertex. Color all edges incident to the first set with $c_1$ and all edges incident to the second set with $c_2$. For the remaining vertex in $V_n$, partition the edges equitably between $c_1$ and $c_2$. Consequently, the color classes contain $\floor{mn/2}$ and $\ceil{mn/2}$ edges, respectively. In each color, the remaining $\floor{n/2}$ vertices in $V_n$ are isolated while the vertices incident to edges of each color form a single nontrivial component. Thus, each color class has $\floor{n/2}+1$ connected components.
\end{construction}

From \eqref{eq:two-colors}, for either parity of $n$, we have $$\wld_{\varepsilon_{2^*}}(K_{m,n})=2\Big(\floor{n/2}+1\Big)-2= 2\floor{n/2}.$$

\begin{figure}[ht]
    \centering
    \begin{minipage}[t]{0.48\textwidth}
    \centering
    \begin{tikzpicture}[baseline=(current bounding box.north), scale=.9, every node/.style={circle, fill=black, inner sep=1.5pt}]
    
    \node[label=left:$a_1$] (a1) at (0,6) {};
    \node[label=left:$a_2$] (a2) at (0,5) {};
    \node[draw=none, fill=none] at (0,4) {$\vdots$};
    \node[label=left:$a_{m-1}$] (am1) at (0,3) {};
    \node[label=left:$a_m$] (am) at (0,2) {};
    
    \node[label=right:$b_1$] (b1) at (4,6) {};
    \node[draw=none, fill=none] at (4,5) {$\vdots$};
    \node[label=right:$b_{n/2}$] (bfloor) at (4,4) {};
    \node[label=right:$b_{n/2+1}$] (bceil) at (4,3) {};
    \node[draw=none, fill=none] at (4,2) {$\vdots$};
    \node[label=right:$b_n$] (bn) at (4,1) {};
    
    \foreach \x in {a1,a2,am1,am}{
      \foreach \y in {b1,bfloor}{
        \draw[red,thick] (\x) -- (\y);
      }
    }
    
    \foreach \x in {a1,a2,am1,am}{
      \foreach \y in {bceil,bn}{
        \draw[blue,thick] (\x) -- (\y);
      }
    }
    
    \node[draw=none, fill=none] at (0,6.7) {$V_m$};
    \node[draw=none, fill=none] at (4,6.7) {$V_n$};
    
    \node[draw=none, fill=none] at (2,7.3) {$K_{m,n}$, $n$ even};
    
    \end{tikzpicture}
\end{minipage}%
\hfill
\begin{minipage}[t]{0.48\textwidth}
    \centering
    \begin{tikzpicture}[baseline=(current bounding box.north), scale=.9, every node/.style={circle, fill=black, inner sep=1.5pt}]
        \node[label=left:$a_1$] (a1) at (0,6) {};
        \node[label=left:$a_2$] (a2) at (0,5) {};
        \node[draw=none, fill=none] at (0,4) {$\vdots$};
        \node[label=left:$a_{m-1}$] (am1) at (0,3) {};
        \node[label=left:$a_m$] (am) at (0,2) {};
        
        \node[label=right:$b_1$] (b1) at (4,6) {};
        \node[draw=none, fill=none] at (4,5) {$\vdots$};
        \node[label=right:$b_{\lceil n/2\rceil}$] (bceil) at (4,4) {};
        \node[draw=none, fill=none] at (4,3) {$\vdots$};
        \node[label=right:$b_n$] (bn) at (4,2) {};
        
        \foreach \x in {a1,a2,am1,am}{
          \foreach \y in {b1}{
            \draw[red,thick] (\x) -- (\y);
          }
        }
        
        \draw[red,thick] (a1) -- (bceil);
        \draw[red,thick] (a2) -- (bceil);
        \draw[blue,thick] (am1) -- (bceil);
        \draw[blue,thick] (am) -- (bceil);
        
        \foreach \x in {a1,a2,am1,am}{
          \foreach \y in {bn}{
            \draw[blue,thick] (\x) -- (\y);
          }
        }
        
        \node[draw=none, fill=none] at (0,6.7) {$V_m$};
        \node[draw=none, fill=none] at (4,6.7) {$V_n$};
        
        \node[draw=none, fill=none] at (2,7.3) {$K_{m,n}$, $n$ odd};
    \end{tikzpicture}
\end{minipage}
    \caption{The equitable $2$-colorings of $K_{m,n}$ attaining $\wld=2\floor{n/2}$.}
    \label{fig:kmn-eq2}
\end{figure}

\begin{theorem}\label{thm:bip-eq2}
    For any equitable $2$-coloring of the edges of $K_{m,n}$ with $2\le m\le n$,
    $$\wld(K_{m,n})\le 2\floor{n/2}.$$ Moreover, this bound is sharp.
\end{theorem}

\begin{proof}
Let $G=K_{m,n}$ with an equitable $2$-coloring. We will prove that $\sum_{i=1}^2 (\kappa(G_i)-1)\le 2\floor{n/2}$. 

If $G_1$ and $G_2$ are connected, then $\sum_{i=1}^2 (\kappa(G_i)-1) =0\le 2\floor{n/2}$. We now consider cases where at least one of $\{G_1,G_2\}$ is disconnected. Let $(V_A)_i\subseteq V_m$ and $(V_B)_i\subseteq V_n$ be the vertices with degree at least one in color $c_i$, with $|(V_A)_i|=A_i$ and $|(V_B)_i| = B_i$. 

    \case{$A_1<m$ and $B_1<n$}\label{case:bip-isolated}
    Then $G_1$ has an isolated vertex in each part; let $v\in V_m$ and $w\in V_n$ be isolated in $G_1$. In $G_2$, $v$ is adjacent to every vertex in $V_n$ and $w$ is adjacent to every vertex in $V_m$ including $v$. Thus, $\kappa(G_2)=1$. It remains to show $\kappa(G_1)-1\le 2\floor{n/2}$. 

    We have $\kappa(G_1)=m - A_1 +n-B_1 + \kappa^*(G_1)$. Let $b=\ceil{|E(G_1)|/A_1}$, the least integer $x$ with $|E(G_1)|\le A_1 x$. Note that $1 \le b \le n$. Consider the bipartite graph $$G_1':=K_{A_1,b} \cup K_{m-A_1,0} \cup K_{0,n-b}\subseteq K_{m,n}.$$ Then, $\kappa(G_1')=m-A_1+n-b+1$. We will show that $\kappa^*(G_1)\le B_1+1-b$, which implies $\kappa(G_1)\le \kappa(G_1')$.
    
    By \cref{lem:bip-max-edges}, $|E(G_1)| \le A_1(B_1-\kappa^*(G_1)+1)$ and so $\frac{|E(G_1)|}{A_1} \le B_1-\kappa^*(G_1)+1$. Since $B_1-\kappa^*(G_1)+1\in \mathbb{N}$, $\ceil{|E(G_1)|/A_1}\le B_1-\kappa^*(G_1)+1$. By definition, $b = \ceil{|E(G_1)|/A_1}$, which gives $b\le B_1-\kappa^*(G_1)+1$, or equivalently $\kappa^*(G_1)\le B_1+1-b$ and thus $\kappa(G_1)\le \kappa(G_1')$.
    
    Now, let $S:=A_1+b=n+m+1-\kappa(G_1')$. Then $b=S-A_1$. We also have $|E(G_1)|\le A_1 b$. Then
    \begin{align*}
    |E(G_1)|\le A_1b=A_1(S-A_1)\le \floor{S^2/4} &= \floor*{\frac{(m+n+1-\kappa(G_1'))^2}
    {4}}\le   \frac{(m+n+1-\kappa(G_1'))^2}
    {4}\\
    4|E(G_1)|&\le (m+n+1-\kappa(G_1'))^2\\
    \kappa(G_1')&\le m+n+1-2\sqrt{|E(G_1)|}\quad\text{since } m+n+1-\kappa(G_1') = S > 0.
    \end{align*}
    Since $|E(G_1)|\in \{\floor{mn/2},\ceil{mn/2}\}$, we then have $$\kappa(G_1') \le m+n+1-2\sqrt{\floor{mn/2}} \le m+n+1-\sqrt{2(mn-1)} \le \begin{cases} m+n+1-m & \text{if } n \text{ is even}\\ 
 m+n+1-(m+1) & \text{if } n \text{ is odd}
\end{cases}\le  2\floor{n/2}+1$$ where the second to last inequality follows from the fact that $\sqrt{2(mn-1)}\ge m$ when $n$ is even and $\sqrt{2(mn-1)}\ge m+1$ when $n$ is odd for $2 \le m \le n$.
      
    Thus, $\kappa(G_1)\le \kappa(G_1')\le 2\floor{n/2}+1$, so $\sum_{i=1}^2 (\kappa(G_i)-1) \le 2\floor{n/2}$. By symmetry, if either $G_1$ or $G_2$ satisfies the conditions of \cref{case:bip-isolated}, we are done.
    
    \case{$A_i=m$ or $B_i=n$ for each $i$}
    Suppose that $A_i=m$. Then \begin{equation}
        \kappa(G_i)=\kappa^*(G_i)+n-B_i \label{eq:bip-kappa}.
    \end{equation} By \cref{lem:bip-max-edges}, $|E(G_i)|\le m(B_i-\kappa^*(G_i)+1).$ Since $|E(G_i)|\ge\floor{mn/2}$,
    \begin{align}
        \floor{mn/2}&\le m(B_i-\kappa^*(G_i)+1)\nonumber\\
        \kappa^*(G_i)&\le B_i+1-\frac{1}{m}\floor{mn/2}\label{eq:bip-kappa-star}
    \end{align} Combining \eqref{eq:bip-kappa} and \eqref{eq:bip-kappa-star}, we have 
    $$
        \kappa(G_i)
        =\kappa^*(G_i)+n-B_i\le 
        B_i+1-\frac{1}{m}\floor{mn/2}+n-B_i
        =
        n+1-\frac{1}{m}\floor{mn/2}
        =
        \begin{cases} n+1 - n/2 & \text{if } mn \text{ even}\\ 
 n+1-n/2+1/(2m) & \text{if } mn \text{ odd}
\end{cases}
    $$ For either parity of $mn$, since $m \ge 2$ and $\kappa(G_i)$ is an integer, we have
    $$
        \kappa(G_i)
        \le 
        \floor{n/2}+1
    $$
    
    Now, suppose that $B_i=n$. A similar argument produces 
    $$
        \kappa(G_i)
        \le
        \floor{m/2}+1
        \le 
        \floor{n/2}+1
    $$ Thus, for $i=1,2$, we have $\kappa(G_i)\le \floor{n/2}+1$ so $\sum_{i=1}^2(\kappa(G_i)-1)\le 2\floor{n/2}$, and we are done. 

    In all cases, we have $\sum_{i=1}^2(\kappa(G_i)-1)\le 2\floor{n/2}$. By \cref{prop:comp-bound}, $$\wld(G)\le 2\floor{n/2}.$$ The sharpness of the bound is guaranteed by \cref{con:eq2bip}.
    \end{proof}
    
\subsection{Minimum Proper Colorings}
By K{\"o}nig's edge-coloring theorem \cite{RD2016}, the complete bipartite graph $K_{m,n}$ satisfies $$\chi'(K_{m,n}) = \Delta(K_{m,n}).$$
For $m\le n$, every minimum proper coloring of $K_{m,n}$ uses exactly $n$ colors. Moreover, each color class is a matching of size $m$. For a color $c_i$, we write $M_i$ for its monochromatic matching. We determine the wild number for minimum proper colorings of $K_{m,n}$ and show that it is independent of the choice of proper edge coloring.

\begin{theorem}\label{thm:bip-proper}
Let $K_{m,n}$ be a complete bipartite graph with $2 \le m \le n$. For any minimum proper coloring $\rho$ of $K_{m,n}$, $$\wld_{\rho}(K_{m,n}) = n.$$
\end{theorem}

\begin{proof}
    Let $M_1,\ldots, M_n$ denote the color classes of $K_{m,n}$ with respect to $\rho$. Each $M_i$ consists of $m$ independent edges, forming $m$ components of size 2. The remaining $n-m$ vertices are isolated, forming $n-m$ components of size 1. Thus, each color class has $m + (n-m) = n$ components. From \cref{prop:ceil-bound}, it follows that $$
        \wld_\rho(K_{m,n})\ge \left\lceil\frac{\sum_{i=1}^n(\kappa(M_i)-1)}{n-1}\right\rceil=\frac{n(n-1)}{n-1}=n.
    $$
    
    To prove the upper bound, take some $v\in V_m$ and let $W$ consist of every edge incident to $v$. Note that $|W|=n$. 

    Consider any color class $M_i$. Since $M_i$ is a matching of size $m$, $M_i$ covers $V_m$. Thus, $M_i$ contains exactly one edge incident to $v$. The remaining $n-1$ edges of $W$ connect $v$ to every vertex in $V_n$ that is not matched to $v$ in $M_i$, including all $n-m$ vertices of $V_n$ isolated in $M_i$. The vertex of $V_n$ matched to $v$ in $M_i$ is already connected to $v$ by its edge in $M_i$.
    
    For any other vertex $v'\in V_m \setminus \{v\}$, $M_i$ matches $v'$ with some vertex $u'\in V_n$. Thus, the edge $u'v'$ is in $M_i$, while the edge $u'v$ belongs to $W$. Hence, $u'$ and $v'$ are connected to $v$ in $M_i\cup W$. Therefore, every vertex of $K_{m,n}$ belongs to the same component of $M_i\cup W$. Since $M_i$ was arbitrary, $W$ is color-connecting. Together with the lower bound, $W$ is ideal and 
    \[
        \wld_{\rho}(K_{m,n})=n.\qedhere
    \]
\end{proof}

\section{Cartesian Products}\label{sec:products}

Throughout this section, we write $\rho$ for an arbitrary minimum proper coloring of $G \Box H$ using $\chi'(G\Box H)$ colors. Recall that $|V(G \Box H)| = |V(G)||V(H)|$ and $|E(G \Box H)| = |E(G)||V(H)|+|V(G)||E(H)|$.

\begin{definition}\label{def:product-coloring}
    Fix minimum proper colorings $\rho_1$ and $\rho_2$ of $G$ and $H$, respectively, with disjoint color sets. We define $\rho^*$ as the edge coloring of $G \Box H$ where $\rho^*((g,h)(g',h)) = \rho_1(gg')$ and $\rho^*((g,h)(g,h')) = \rho_2(hh')$.
\end{definition} This coloring is proper and uses $\chi'(G) + \chi'(H)$ colors, but, in general, is not minimum.

\begin{lemma}\label{lem:product}
    For a minimum proper coloring of $G \Box H$, there exists a color $c_i$ such that $(G\Box H)_i$ has at least
    \[
    |V(G \Box H)|
    -
    \frac{|E(G \Box H)|}{\chi'(G \Box H)}
    \]
    connected components.
\end{lemma}

\begin{proof}
    Let $c_1,\dots,c_{\chi'(G\Box H)}$ be the colors of the minimum proper coloring and let $m_i$ be the number of edges colored $c_i$. Each color class is a matching, so $\kappa((G \Box H)_i) = m_i+(|V(G \Box H)|-2m_i) = |V(G \Box H)| - m_i$ where
    $$ 
        \sum_{i=1}^{\chi'(G \Box H)} m_i = |E(G \Box H)|.
    $$
    By averaging, there exists some color $c_i$ such that $m_i \le |E(G \Box H)|/\chi'(G \Box H)$. Therefore, 
    $$
        \kappa((G\Box H)_i) \ge |V(G\Box H)| - \frac{|E(G \Box H)|}{\chi'(G \Box H)}.\qedhere
    $$
\end{proof}

The bounds in this section are obtained by bounding from below the number of components of some color class. For an $\ell$-edge-colored graph $(G, \gamma)$, we write $$\clb_{\gamma}(G)=\max_{1 \le i \le \ell} (\kappa(G_i) - 1).$$ Note that $\clb_{\gamma}(G) \le \wld_{\gamma}(G)$ by \cref{prop:comp-bound}.

\subsection{Complete Graphs}

\begin{proposition}\label{prop:product-kn}
    Let $n,m \ge 4$. Let $\rho_1$ and $\rho_2$ be any minimum proper colorings of $K_n$ and $K_m$, respectively. Then, for $\rho' \in \{\rho, \rho^*\}$, $$\wld_{\rho_1}(K_n) + \wld_{\rho_2}(K_m) < \wld_{\rho'}(K_n \Box K_m).$$
\end{proposition}
\begin{proof}
    We consider the cases $\rho'=\rho$ and $\rho'=\rho^*$. By definition, $\wld_{\rho_1}(K_n) \le n-1$ and $\wld_{\rho_2}(K_m) \le m-1$.
    
    \case{$\rho'=\rho$}
    We show that $\clb_{\rho'}(K_n \Box K_m) > \wld_{\rho_1}(K_n) + \wld_{\rho_2}(K_m).$ By \cref{lem:product}, there exists a color whose corresponding subgraph has at least
    \begin{align*}
        |V(K_n \Box K_m)| - \frac{|E(K_n \Box K_m)|}{\chi'(K_n \Box K_m)}.
    \end{align*}
    components. Since $\Delta(K_n \Box K_m) = n+m-2$, we have $\chi'(K_n \Box K_m) \ge n+m-2$. Combining \cref{lem:product} with the definition of  $\clb_{\gamma}(G)$, we then have
    \begin{align*}
        \frac{nm}{2} - 1 = nm - \frac{\binom{n}{2} m + \binom{m}{2} n}{n + m - 2} - 1 &\le nm - \frac{\binom{n}{2} m + \binom{m}{2} n}{\chi'(K_n\Box K_m)} - 1 \le \clb_{\rho'}(K_n \Box K_m).
    \end{align*}
    Observe that $n + m - 2 < nm/2-1$ since $n,m \ge 4$, thus $$\wld_{\rho_1}(K_n) + \wld_{\rho_2}(K_m) \le n + m - 2 < \frac{nm}{2} - 1 \le \clb_{\rho'}(K_n \Box K_m) \le \wld_{\rho'}(K_n \Box K_m).$$
    
    \case{$\rho'=\rho^*$}
 Again, we show that $\clb_{\rho'}(K_n \Box K_m) > \wld_{\rho_1}(K_n) + \wld_{\rho_2}(K_m)$. Recall that when $n$ is even each color class of $K_n$ is a perfect matching of size $n/2$ and when $n$ is odd each color class of $K_n$ is a near-perfect matching of size $(n - 1)/2$. Thus, when $n$ is even the subgraph corresponding to each color class of $K_n$ contains $n/2$ components and when $n$ is odd the subgraph corresponding to each color class of $K_n$ contains $(n + 1)/2$ components. Under $\rho^*$, the subgraph associated with a color class from $\rho_1$ occurs in each of the $m$ copies of $K_n$, so when $n$ is even $\clb_{\rho'}(K_n \Box K_m) \ge nm/2 - 1$ and when $n$ is odd $\clb_{\rho'}(K_n \Box K_m) \ge m (n+1)/2 - 1$. Observe that $n + m - 2 < nm/2-1$ and $n + m - 2 < m(n+1)/2-1$ since $n,m \ge 4$. For either parity of $n$, we then have 
    \[
        \wld_{\rho_1}(K_n) + \wld_{\rho_2}(K_m) <\clb_{\rho'}(K_n \Box K_m) \le \wld_{\rho'}(K_n \Box K_m).\qedhere
    \]
\end{proof}

\subsection{Complete Bipartite Graphs}

\begin{proposition}\label{prop:product-kmn}
    Let $2\le m\le n$ and $2\le s\le t$. Let $\rho_1$ and $\rho_2$ be any minimum proper colorings of $K_{m,n}$ and $K_{s,t}$, respectively. Then, for $\rho' \in\{\rho, \rho^*\}$, $$\wld_{\rho_1}(K_{m,n}) \cdot \wld_{\rho_2}(K_{s,t}) < \wld_{\rho'}(K_{m,n} \Box K_{s,t}).$$
\end{proposition}

\begin{proof}
We consider the cases $\rho'=\rho$ and $\rho'=\rho^*$. From \cref{thm:bip-proper}, we have $\wld_{\rho_1}(K_{m,n})= n $ and $\wld_{\rho_2}(K_{s,t}) =t$.
    
    \case{$\rho'=\rho$}
    Note that the Cartesian product of bipartite graphs is again bipartite, so $\chi'(K_{m,n} \Box K_{s,t}) = \Delta(K_{m,n} \Box K_{s,t}) =\Delta(K_{m,n})+\Delta(K_{s,t})=n+t$ by K{\"o}nig's edge-coloring theorem. Applying \cref{lem:product},
    \begin{align*}
        (s + t)(m + n) - \frac{mn(s+t)+st(m+n)}{n+t}-1 
        &\le \clb_{\rho'}(K_{m,n}\Box K_{s,t}).
    \end{align*}
    We show that $\wld_{\rho_1}(K_{m,n}) \cdot \wld_{\rho_2}(K_{s,t}) = n t < (s + t)(m + n) - \frac{mn(s+t)+st(m+n)}{n+t}-1.$ 

    Observe that $$(s + t)(m + n) - \frac{mn(s+t)+st(m+n)}{n+t}-1 - (nt) = \frac{sn^2+mt^2-n-t}{n+t},$$ so it is enough to show that $sn^2+mt^2-n-t = n(sn-1)+t(mt-1)>0$ which is true for $m,n,s,t \ge 2$. It follows that $$\wld_{\rho_1}(K_{m,n}) \cdot \wld_{\rho_2}(K_{s,t}) < \clb_{\rho'}(K_{m,n}\Box K_{s,t}) \le  \wld_{\rho'}(K_{m,n} \Box K_{s,t}).$$

    \case{$\rho'=\rho^*$}
    Recall that a minimum proper coloring of $K_{m,n}$ uses $n$ colors where each color class contains $m$ edges, so the corresponding subgraph contains $m+(n-m)=n$ components for each color class. Under $\rho^*$, the subgraph associated with a color class from $\rho_1$ occurs in each of the $s+t$ copies of $K_{m,n}$, so $\clb_{\rho'}(K_{m,n} \Box K_{s,t}) \ge n(s+t)-1$. Thus, $$\wld_{\rho_1}(K_{m,n}) \cdot \wld_{\rho_2}(K_{s,t}) = nt < n(s + t) - 1 \le \wld_{\rho'}(K_{m,n} \Box K_{s,t})$$
    where the first inequality follows from the fact that $ns>1$ for $n,s \ge 2$. 

\end{proof}

\section{Concluding Remarks}\label{sec:conclusion}

\cref{thm:small-complete-prop-nums} shows that the wild number for $K_n$ with a minimum proper coloring is exactly the lower bound for $3 \le n \le 12$. These results lead to the following conjecture. 
\begin{conjecture}
\label{conj:proper-kn-exact}
For every $n\ge3$ and every minimum proper coloring $\rho$ of $K_n$,
$$
    \wld_\rho(K_n)
    =
    \left\lfloor
        \frac{4\lceil n/2\rceil}{3}
    \right\rfloor.
$$
\end{conjecture}

Question~7.5 of \cite{WildNumber} asks how far the greedy algorithm can deviate from the wild number. For a minimum proper coloring $\rho$ of $K_n$, we showed in \cref{lem:greedy-acyclic} that the greedy output is always a forest and therefore has at most $n-1$ edges. Combining this with \cref{thm:proper-kn-lowerbound} gives 
$$ 
    0 
    \le 
    |\Wg|-\wld_\rho(K_n) 
    \le 
    n-1-\left\lfloor
        \frac{4\lceil n/2\rceil}{3}
    \right\rfloor. 
$$

If the algorithm saturates a color class, then its output is a spanning tree and therefore has exactly $n-1$ edges. By \cref{cor:greedy-saturation}, this output is not ideal whenever $n\ge 7$, since smaller color-connecting sets exist by \cref{thm:proper-odd-upper,cor:proper-even-upper}. 

\section*{Acknowledgments}

This research was conducted during the summer of 2026 as part of the Summer Undergraduate Applied Mathematics Institute at Carnegie Mellon University. The first author is supported by NSF Grant DMS-2536176. 

\bibliographystyle{plainurl}
\bibliography{references}

\end{document}